\documentclass{article}
\usepackage{amsmath}
\usepackage{amscd,amsthm}
\usepackage{amstext}
\usepackage{amssymb}
\usepackage{amsfonts}
\usepackage{latexsym}
\usepackage{bookman}
\usepackage{tikz-cd}
\usepackage{color}
\usepackage{cite}
\numberwithin{equation}{section}
\allowdisplaybreaks[1]
\newtheorem{theorem}{Theorem}[section] 
\newtheorem{lemma}[theorem]{Lemma}
\newtheorem{proposition}[theorem]{Proposition} 
\newtheorem{corollary}[theorem]{Corollary} 
\theoremstyle{definition}
\newtheorem{definition}[theorem]{Definition} 
\theoremstyle{remark}
\newtheorem{ownremark}[theorem]{Remark} 
\newcommand{\beq}{\begin{eqnarray}}
\newcommand{\eeq}{\end{eqnarray}}
\newcommand{\hra}{\hookrightarrow}
\newcommand{\bli}{\begin{list}{}{\labelwidth6mm\leftmargin8mm}}
\newcommand{\eli}{\end{list}}
\newcommand{\ds}{\displaystyle}
\newcommand{\R}{{\mathbb R}}
\newcommand{\Rn}{{\mathbb R}^{d}} 
\newcommand{\N}{\ensuremath{\mathbb N}}
\newcommand{\No}{\N_{0}}
\newcommand{\nat}{\N}
\newcommand{\no}{\No}
\newcommand{\Z}{\mathbb Z} 
\newcommand{\Zn}{\Z^{d}}
\newcommand{\C}{{\mathbb C}}
\newcommand{\SRn}{\mathcal{S}(\Rn)}
\newcommand{\SpRn}{\mathcal{S}'(\Rn)}

\newcommand{\dint}{\;\mathrm{d}}
\newcommand{\dist}{\;\mathrm{dist}}
\newcommand{\A}{\ensuremath{A^s_{p,q}}}  
\newcommand{\B}{\ensuremath{B^s_{p,q}}}  
\newcommand{\be}{\ensuremath{B^{s_1}_{p_1,q_1}}}  
\newcommand{\bz}{\ensuremath{B^{s_2}_{p_2,q_2}}}  
\newcommand{\fe}{\ensuremath{F^{s_1}_{p_1,q_1}}}  
\newcommand{\fz}{\ensuremath{F^{s_2}_{p_2,q_2}}}  
\newcommand{\Ae}{\ensuremath{A^{s_1}_{p_1,q_1}}}  
\newcommand{\Az}{\ensuremath{A^{s_2}_{p_2,q_2}}}  
\newcommand{\F}{\ensuremath{F^s_{p,q}}}  
\newcommand{\supp}{\ensuremath{{\mathop{\rm supp}\nolimits\, }}}
\def\Id{\mathop{\rm Id}\nolimits}
\def\id{\mathop{\rm id}\nolimits}
\newcommand{\cl}{{\mathcal{L}}} 
\newcommand{\lrie}{L_{r,\infty}^{(e)}}
\newcommand{\whole}[1]{\ensuremath\left\lfloor #1 \right\rfloor}
\newcommand{\M}{\ensuremath{{\cal M}_{u,p}}}  
\newcommand{\MA}{\ensuremath{{\mathcal A}}^{s}_{u,p,q}}  
\newcommand{\MB}{\ensuremath{{\cal N}^{s}_{u,p,q}}}  
\newcommand{\MBe}{\ensuremath{{\cal N}^{s_1}_{u_1,p_1,q_1}}}  
\newcommand{\MBz}{\ensuremath{{\cal N}^{s_2}_{u_2,p_2,q_2}}}  
\newcommand{\MFe}{\ensuremath{{\cal E}^{s_1}_{u_1,p_1,q_1}}}  
\newcommand{\MFz}{\ensuremath{{\cal E}^{s_2}_{u_2,p_2,q_2}}}
\newcommand{\MAe}{\ensuremath{{\cal A}^{s_1}_{u_1,p_1,q_1}}}  
\newcommand{\MAz}{\ensuremath{{\cal A}^{s_2}_{u_2,p_2,q_2}}}  
\newcommand{\MF}{\ensuremath{{\cal E}^{s}_{u,p,q}}}  
\newcommand{\mb}{\ensuremath{n^{s}_{u,p,q}}}

\newcommand{\msib}{\ensuremath{n^{\sigma}_{u,p,q}}}

\newcommand{\mbt}{\ensuremath{\widetilde{n}^{\sigma}_{u,p,q}}}  
\newcommand{\mbet}{\ensuremath{\widetilde{n}^{\sigma_1}_{u_1,p_1,q_1}}}  
\newcommand{\mbzt}{\ensuremath{\widetilde{n}^{\sigma_2}_{u_2,p_2,q_2}}}  
\newcommand{\mmb}{\ensuremath{m^{2^{jd}}_{u,p}}} 
\newcommand{\mmbet}{\ensuremath{{m}^{2^{jd}}_{u_1,p_1}}}  
\newcommand{\mmbzt}{\ensuremath{{m}^{2^{jd}}_{u_2,p_2}}}  
\newcommand{\ext}{\ensuremath{{\mathop{\rm ext}\nolimits}}}
\newcommand{\re}{\ensuremath{\mathop{\rm re}\nolimits}}
\newcommand{\ignore}[1]{}
\title{ Entropy numbers of compact embeddings of Smoothness Morrey spaces on bounded domains}
\author{Dorothee~D.~Haroske\thanks{Institute of Mathematics, Friedrich Schiller University Jena, 07737 Jena, Germany}, 
 and Leszek~Skrzypczak\thanks{Faculty of Mathematics and  Computer Science, Adam Mickiewicz University,  61-614 Poznan, Poland} \thanks{The author  was supported by National Science Center, Poland,  Grant No.2013/10/A/ST1/00091.}}% \footnotemark[1]}%\;\footnotemark[2]}
\begin{document} 
\maketitle 
%\end{document}
\begin{abstract} 
We study the compact embedding between smoothness Morrey spaces on bounded domains and characterise its entropy numbers. Here we discover a new phenomenon when the difference of smoothness parameters in the source and target spaces is rather small compared with the influence of the fine parameters in the Morrey setting. In view of some partial forerunners this was not to be expected till now. Our argument relies on wavelet decomposition techniques of the function spaces and a careful study of the related sequence space setting. 
\end{abstract} 
%\footnotetext[1 ]{The author  was supported by National Science Center, Poland,  Grant No.2013/10/A/ST1/00091.}
%\footnotetext[2]{Faculty of Mathematics and  Computer Science, Adam Mickiewicz University,  61-614 Pozna\'n, Poland}
%\input{newcommn.tex} 
%%%%%%%%%%%%%%%%%%%%%%%%%%%%%%%%%%%%%%%%%%%
%Programme
%(1) 
%(2) 
%(3) 
%(4) 
%%%%%%%%%%%%%%%%%%%%%%%%%%%%%%%%%%%%%%%%%%
%%%%%%%%%%%%%%%%%%%%%%%%%%%%%%%%%%%%%%%\"
%\bli
%\item[{\upshape\bfseries (a)}]
%$\MBe(\R^d)\, =\, \MBz(\R^d)$ (with norm equivalence), 
%\item[{\upshape\bfseries (b)}]
%$(s_1,p_1,q_1,r_1)\,=\,(s_2,p_2,q_2,r_2)$. 
%\el
%%%%%%%%%%%%%%%%%%%%%%%%%%%%%%%%%%%%%%%%%%%%%%%%%%%%%5
%%%%%%%%%%%%%%%%%%%%%%%%%%%%%%%%%%%%%%%%%%%%%%%%%%%5

\section{Introduction}
Let $\Omega\subset\Rn$ be a bounded $C^\infty$ domain and $\MB(\Omega)$ and $\MF(\Omega)$ smoothness Morrey spaces, with $s_i\in\R$, $0<p_i\leq u_i<\infty$, $0<q_i\leq\infty$, $i=1,2$. Roughly speaking, these spaces $\MB$ and $\MF$ are the counterparts of the well-known Besov and Triebel-Lizorkin function spaces $\B$ and $\F$, respectively, where the basic $L_p$ space in the latter scales is replaced by the {Morrey space}
  $\M$, $0<p\le u<\infty $: this is the set of all
  locally $p$-integrable functions $f\in L_p^{\mathrm{loc}}(\Rn)$  such that
\begin{equation}\label{i-Mo}
\|f \mid {\M(\Rn)}\| =\, \sup_{x\in \Rn, R>0} R^{\frac{d}{u}-\frac{d}{p}}
\left[\int_{B(x,R)} |f(y)|^p \dint y \right]^{\frac{1}{p}}\, <\, \infty\, .
\end{equation}
%Consequently $\mathcal{M}_{p,p}=L_p$ and likewise $\B=\mathcal{N}^s_{p,p,q}$,  $\F=\mathcal{E}^s_{p,p,q}$.
For the precise definition and further properties we refer to Section~\ref{sect-MB-Rn} below.

In our recent papers \cite{hs12b,hs14} we obtained a complete characterisation when the corresponding embeddings
\begin{equation}\label{i-0}
\id_{\mathcal{N}}: \MBe(\Omega)\to \MBz(\Omega)\quad \text{and}\quad \id_{\mathcal{E}}:\MFe(\Omega)\to \MFz(\Omega)
\end{equation}
are compact. The main purpose of the present paper is to characterise this compactness in terms of the corresponding entropy numbers. Apart from our first outcome in \cite{HaSk-krakow} (and some parallel results for approximation numbers in \cite{YHMSY,Bai-Si}) we are not aware of any other such investigation. This seems a little surprising in view of the popularity of smoothness Morrey spaces recently and the possible applications of entropy (and approximation) number results. However, these latter options are out of the scope of the present paper.

Smoothness Morrey spaces have been studied  intensely in the past couple of years opening a wide field of possible applications. 
The Besov-Morrey spaces $\mathcal{N}^{s}_{u,p,q}(\Rn)$ were introduced  in \cite{KY} by Kozono and Yamazaki and used by them and Mazzucato \cite{Maz} to study Navier-Stokes equations. Corresponding Triebel-Lizorkin-Morrey spaces $\mathcal{E}^s_{u,p,q}(\Rn)$ were introduced in \cite{TX} by Tang and Xu, where the authors established the Morrey version of Fefferman-Stein vector-valued inequality. Properties of these spaces such as wavelet characterisations were studied in the papers by Sawano \cite{Saw2, Saw1}, Sawano and Tanaka \cite{ST2, ST1}, and Rosenthal \cite{MR-1}.   We give some further references in Section~\ref{sect-MB-Rn} below.

Now we concentrate on embeddings within the scale of such spaces. Let $s_i\in\R$, $0<p_i\leq u_i<\infty$, $0<q_i\leq\infty$, $i=1,2$. In \cite{HaSk-bm1,hs14} we characterised the continuity of the embeddings 
%\[
$\id: \MBe(\Rn)\to \MBz(\Rn)$ \text{and} $\id:\MFe(\Rn)\to \MFz(\Rn)$.
%\]
Furthermore, in \cite{HaSk-bm1, hs12b, hs14} we studied some limiting embeddings.
But these embeddings can never be compact.
However, turning to spaces on bounded domains, we obtained in \cite{hs12b,hs14} necessary and sufficient conditions for the corresponding embeddings \eqref{i-0} to be compact:  this is the case if, and only if,
\begin{equation}\label{i-1}
\frac{s_1-s_2}{d} > \max\left\{0,\frac{1}{u_1}-\frac{1}{u_2}, \frac{p_1}{u_1}\left(\frac{1}{p_1}-\frac{1}{p_2}\right)\right\}.
\end{equation}
Recall that in case of $p_i=u_i$, $i=1,2$, we return to the classical situation of Besov and Triebel-Lizorkin spaces. %, since $\B = \mathcal{N}^s_{p,p,q}$ and  $\F = \mathcal{E}^s_{p,p,q}$.
 Then the above findings \eqref{i-1} are in perfect agreement with { the well-known condition  $\frac{s_1-s_2}{d} > \max\left\{0,\frac{1}{p_1}-\frac{1}{p_2}\right\}$,}
%that  well-known situation where 
%\begin{align}
%  \id_B: \be(\Omega)\to\bz(\Omega) \quad \text{compact}&\nonumber\\
%  \quad \text{if, and only  if,}\quad & \frac{s_1-s_2}{d} > \max\left\{0,\frac{1}{p_1}-\frac{1}{p_2}\right\},\label{i-4}
%\end{align}
for the compactness of $\id_B: \be(\Omega)\to\bz(\Omega)$, 
similarly for $\id_F: \fe(\Omega)\to\fz(\Omega)$. 
However, a lot more about the compactness of the embeddings $\id_B$, $\id_F$ has been found in that `classical' situation. The `degree of compactness' as reflected by the asymptotic behaviour of the corresponding entropy numbers, is well-known; for the definition of  entropy numbers as well as further applications of this concept we refer to Section~\ref{sect-ek} below. Concerning the embedding $\id_B$ Edmunds and Triebel obtained in  \cite{ET1,ET2} for the asymptotic decay of the entropy numbers that 
\begin{equation}
\label{i-3}
e_k\left(\id_B : \be(\Omega)\to \bz(\Omega)\right)\quad\sim
\quad k^{-\frac{s_1-s_2}{d}},\quad k\in\N,
\end{equation}
the result for $F$-spaces is  parallel.

Now we tackle the Morrey situation in full generality, that is, we study the asymptotic behaviour of the entropy numbers of the compact embeddings $\id_{\mathcal{N}}$ and  $\id_{\mathcal{E}}$ given by \eqref{i-0} whenever \eqref{i-1} is satisfied. In our recent contribution  \cite{HaSk-krakow} we obtained first results in some cases, using sharp  embeddings and interpolation techniques as main tools. In particular, under the additional assumption {$\frac{s_1-s_2}{d} > \max\left\{0, \frac{1}{p_1}-\frac{1}{u_2}\right\}$}  
%  \begin{equation}\label{ek-1}
%\frac{s_1-s_2}{d} > \max\left\{0, \frac{1}{p_1}-\frac{1}{u_2}\right\}
%\end{equation}
we could prove {the same asymptotic behaviour of the entropy numbers   of  $\id_{\mathcal{N}}$ as in \eqref{i-3}},
%that
%\begin{equation}\label{ek-2} 
%	e_k\left( \id_{\mathcal{N}}: \MBe(\Omega )\hookrightarrow \MBz(\Omega )\right)\sim k^{-\frac{s_1-s_2}{d}},\quad k\in\N,
%\end{equation}
with the observation, that the lower estimate remains valid in  all cases of compactness \eqref{i-1}. Again, $\id_{\mathcal{N}}$ can be replaced by $\id_{\mathcal{E}}$. %in \eqref{ek-2}. 
At this point the remaining case was left open, that is, when
\[
\max\left\{0,\frac{1}{u_1}-\frac{1}{u_2}, \frac{p_1}{u_1}\left(\frac{1}{p_1}-\frac{1}{p_2}\right)\right\}<
\frac{s_1-s_2}{d} \leq \max\left\{0, \frac{1}{p_1}-\frac{1}{u_2}\right\}
\]
is satisfied. Plainly, this case requires at least the source space to be of `proper' Morrey type, that is, $p_1<u_1$. Otherwise it would be excluded. Based on the `classical' outcome \eqref{i-3}, where always the difference of smoothness parameters $s_1-s_2$ characterises the decay of the entropy numbers, as well as the observation that our first approach in \cite{HaSk-krakow} verified the lower bound of that type always to hold, some reasonable assumption was that {\eqref{i-3}} %\eqref{ek-2} 
should be extended to all cases admitted by \eqref{i-1}. However, our present results disprove this claim: we found that the splitting point for the asymptotic behaviour of entropy numbers is $s_1-s_2=d(\frac{1}{p_1}-\frac{1}{p_2})$, which is particularly remarkable since the $p_i$-parameters can be regarded as the fine tuning of the local integrability of the function $f$, recall \eqref{i-Mo}. So the interplay between smoothness and local regularity seems more important than expected so far, since in the local-global result \eqref{i-3} this difference was hidden. Our main outcome is the following:
\begin{itemize}
\item[{\bfseries (a)}]
  If $\quad 
  %\[
  \frac{s_1-s_2}{d} > \max\left\{0,\frac{1}{u_1}-\frac{1}{u_2}, \frac{p_1}{u_1}\left(\frac{1}{p_1}-\frac{1}{p_2}\right)\right\} = \max\left\{0,\frac{1}{u_1}-\frac{1}{u_2}\right\}$, 
  %\]
or
%  \[
$\ \frac{s_1-s_2}{d} >\frac{p_1}{u_1}\left(\frac{1}{p_1}-\frac{1}{p_2}\right) > \max\left\{0,\frac{1}{u_1}-\frac{1}{u_2}\right\}\quad\text{and}\quad \frac{s_1-s_2}{d}>\frac{1}{p_1}-\frac{1}{p_2}$,
%  \]
  then
\begin{equation}\label{ek-3}
  e_k\left( \id_{\mathcal{N}}: \MBe(\Omega )\hookrightarrow \MBz(\Omega )\right)\sim k^{-\frac{s_1-s_2}{d}},\quad k\in\N.
\end{equation}
\item[{\bfseries (b)}]
  If $\ 
  %\[
  \max\left\{0,\frac{1}{u_1}-\frac{1}{u_2}\right\}< \frac{p_1}{u_1}\left(\frac{1}{p_1}-\frac{1}{p_2}\right) < \frac{s_1-s_2}{d} \leq \frac{1}{p_1}-\frac{1}{p_2}$, %\]
  then there exists some $c>0$ and for any $\varepsilon>0$ some $c_\varepsilon>0$ such that for all $k\in\nat$, 
\begin{equation}\label{ek-4}
  c k^{- \alpha} \le    e_k \big(\id_{\mathcal{N}}\big) \le c_\varepsilon k^{- \alpha  +\varepsilon}
\end{equation}
with
%\begin{equation*}
$\ \alpha = \frac{u_1}{u_1-p_1}\left(\frac{s_1-s_2}{d}-\frac{p_1}{u_1}\left(\frac{1}{p_1}- \frac{1}{p_2}\right) \right)$.
%\end{equation*}
  \end{itemize}
  Despite the (only) two-sided estimate in \eqref{ek-4} the lower estimate already disproves the idea that \eqref{ek-3} could be true in all cases. This is in our opinion the really astonishing part of the outcome, but also the other cases (beyond {the one}  %\eqref{ek-1} 
  dealt with in \cite{HaSk-krakow}) are new. {Moreover we prove that the above estimates hold for the smoothness Morrey spaces defined on an  arbitrary bounded domain in $\R^d$.}

          The paper is organised as follows. In Section~\ref{preli} we briefly introduce the corresponding function spaces and the concept of entropy numbers, in Section~\ref{entro_seq} the main results are proved. This is done in terms of appropriate sequence spaces which result from the decomposition of the function spaces by wavelets. In the concluding Section~\ref{entro_func} we establish the results for function spaces and  discuss some special situations.

      \section{Preliminaries}
\label{preli}

First we fix some notation. By $\N$ we denote the \emph{set of natural numbers},
by $\No$ the set $\N \cup \{0\}$. For $a\in\R$, let $a_+:=\max\{a,0\}$ and $\whole{a}=\max\{k\in\Z: k\leq a\}$. For a set $\Omega\subset\Rn$ we denote its characteristic function by $\chi_\Omega$. 

All unimportant positive constants will be denoted by $c$, occasionally with
subscripts. The notation $A \lesssim B$ means that there exists a positive constant $c$ such that
$A \le c \,B$, whereas  the symbol $A \sim B$ stands for $A \lesssim B \lesssim A$.

Given two (quasi-)Banach spaces $X$ and $Y$, we write $X\hookrightarrow Y$
if $X\subset Y$ and the natural embedding of $X$ into $Y$ is continuous.

\subsection{Smoothness Morrey spaces on $\Rn$}\label{sect-MB-Rn}

Recall that the \emph{Morrey space}
  $\M(\Rn)$, $0<p\le u<\infty $, is defined to be the set of all
locally $p$-integrable functions $f\in L_p^{\mathrm{loc}}(\Rn)$  satisfying \eqref{i-Mo}.
%uch that
%$$
%\|f \mid {\M(\Rn)}\| =\, \sup_{x\in \Rn, R>0} R^{\frac{d}{u}-\frac{d}{p}}
%\left[\int_{B(x,R)} |f(y)|^p \dint y \right]^{\frac{1}{p}}\, <\, \infty\, .
%$$

\begin{ownremark}
The spaces $\M(\Rn)$ are quasi-Banach spaces (Banach spaces for $p \ge 1$).
They originated from Morrey's study on PDE (see \cite{Mor}) and are part of the wider class of Morrey-Campanato spaces; cf. \cite{Pee}. They can be considered as a complement to $L_p$ spaces, since $\mathcal{M}_{p,p}(\Rn) = L_p(\Rn)$ with $p\in(0,\infty)$, extended by $\mathcal{M}_{\infty,\infty}(\Rn)  = L_\infty(\Rn)$. In a parallel way one can define the spaces $\mathcal{M}_{\infty,p}(\Rn)$, $p\in(0, \infty)$, but using the Lebesgue differentiation theorem, one arrives at $\mathcal{M}_{\infty, p}(\Rn) = L_\infty(\Rn)$. Moreover, $\M(\Rn)=\{0\}$ for $u<p$, and for  $0<p_2 \le p_1 \le u < \infty$,
%\begin{equation*} 
$	L_u(\Rn)= \mathcal{M}_{u,u}(\Rn) \hookrightarrow  \mathcal{M}_{u,p_1}(\Rn)\hookrightarrow  \mathcal{M}_{u,p_2}(\Rn)$.
%\end{equation*}
\end{ownremark}

Now we present the smoothness spaces of Morrey type in which we are  interested. 
The Schwartz space $\SRn $ and its dual $ \SpRn $ of all 
complex-valued tempered distributions have their usual meaning here.
Let  $ \varphi_0=\varphi \in \SRn $ be such that  
%\[
$\supp \varphi\subset\left\{y\in\Rn:|y|<2\right\}$ {and} $\varphi(x)=1$ {if} $|x|\leq 1$,
%\]
and for each $ j\in\N\;$ let $ \varphi_j(x)=
\varphi(2^{-j}x)-\varphi(2^{-j+1}x)$. Then $ \{\varphi_j\}_{j=0}^\infty $
forms a {\em smooth dyadic resolution of unity}. Given any $ f\in \SpRn$, we denote by 
$ {\mathcal F} f $ and $ {\mathcal F}^{-1} f$ its Fourier transform and its
inverse Fourier transform, respectively.
%Let $f\in \SpRn$, then the Paley-Wiener-Schwartz theorem implies that $\mathcal{F}^{-1}(\varphi_j\mathcal{F}f)$ is an entire analytic function on $\Rn$.

%%%%%%%%%%%%%%%%%%%%%%%%%%%
\begin{definition}\label{d2.5}
Let $0 <p\leq  u<\infty$ or $p=u=\infty$. Let  $0<q\leq\infty$, $s\in \R$ and $\left\{\varphi_j\right\}_j$
a smooth dyadic resolution of unity.
\bli
\item[{\upshape\bfseries (i)}]
The  {\em Besov-Morrey   space}
  $\MB(\Rn)$ is defined to be the set of all distributions $f\in \SpRn$ such that
\begin{align*}
\big\|f\mid \MB(\Rn)\big\|:=
\bigg[\sum_{j=0}^{\infty}2^{jsq}\big\| \mathcal{F}^{-1}(\varphi_j\mathcal{F}f) \mid
\M(\Rn)\big\|^q \bigg]^{1/q} < \infty
\end{align*}
with the usual modification made in case of $q=\infty$.
\item[{\upshape\bfseries  (ii)}]
Let $u\in(0,\infty)$. The  {\em Triebel-Lizorkin-Morrey  space} $\MF(\Rn)$
is defined to be the set of all distributions $f\in   \SpRn$ such that
\begin{align*}
\big\|f \mid \MF(\Rn)\big\|:=\bigg\|\bigg[\sum_{j=0}^{\infty}2^{jsq} |
 \mathcal{F}^{-1}(\varphi_j\mathcal{F}f)|^q\bigg]^{1/q}
\mid \M(\Rn)\bigg\| <\infty
\end{align*}
with the usual modification made in case of  $q=\infty$.
\eli
%%%%%%%%%%%%%%%%%%%%%%%%%%%%%%%%%%%%%%%%%%%%%%%%
\end{definition}

%%%%%%%%%%%%%%%%%%%%%%%%%%%%%%%%%%%

\noindent{\em Convention.}~We adopt the nowadays usual custom to write $\A$ instead of $\B$ or $\F$,  and $\MA$ instead of $\MB$ or $\MF$, respectively, when both scales of spaces are meant simultaneously in some context.

\begin{ownremark}
The  spaces $\MA(\Rn)$ are independent of the particular choice of the smooth dyadic resolution of unity
$\left\{\varphi_j\right\}_j $ appearing in their definitions.
They are quasi-Banach spaces
(Banach spaces for $p,\,q\geq 1$), and $\SRn \hookrightarrow
\MA(\Rn) \hookrightarrow \SpRn$.  When $u=p$, they coincide with the usual  Besov and Triebel-Lizorkin spaces, ${\cal A}^{s}_{p,p,q}(\Rn) = \A(\Rn)$.
There exists extensive literature on the latter spaces; we
refer, in particular, to the series of monographs \cite{T-F1,T-F2,T-F3} for a
comprehensive treatment.
In case of  $u<p$ we have $\MA(\Rn)=\{0\}$. We 
occasionally benefit from the elementary embeddings
\begin{equation}\label{elem-class} 
B^{s}_{p,\min\{p,q\}}(\Rn)\, \hookrightarrow \, \F(\Rn)\, \hookrightarrow \, B^{s}_{p,\max\{p,q\}}(\Rn),
\end{equation}
where $p\in(0,\infty)$, $q\in(0,\infty]$ and $s\in\R$. The result for spaces $\MA$ is different: Sawano proved in \cite{Saw2} that, for $s\in\R$ and $0<p< u<\infty$,
\begin{equation}\label{elem}
	{\cal N}^s_{u,p,\min\{p,q\}}(\Rn)\, \hookrightarrow \, \MF(\Rn)\, \hookrightarrow \,{\cal N}^s_{u,p,\infty}(\Rn),
\end{equation}
where, for the latter embedding, $r=\infty$ cannot be improved.
Mazzucato has shown in \cite[Prop.~4.1]{Maz} that
%\begin{equation}
$\ \mathcal{E}^0_{u,p,2}(\Rn)=\M(\Rn),\quad 1<p\leq u<\infty$.%\label{E-Mup}
%\end{equation}
\end{ownremark}

\begin{ownremark}
Other important, closely related scales of spaces are the so-called Besov-type spaces $B^{s,\tau}_{p,q}(\Rn)$ and Triebel-Lizorkin-type spaces $F^{s,\tau}_{p,q}(\Rn)$, introduced in \cite{ysy}, as well as the local and hybrid spaces dealt with in \cite{t13} and \cite{t14}, { cf. also \cite{s011, s011a} and  \cite{YHMSY, YHSY}.}
% Both subjects are meanwhile studied in great detail, with interesting applications. But this will be out of the scope of the %present paper. We refer to the above monographs as well as to the fine surveys by Sickel \cite{s011, s011a} and the recent %papers \cite{YHMSY, YHSY}.
\end{ownremark}

We briefly recall the wavelet characterisation of Besov-Morrey  spaces  proved in \cite{Saw2}. It will be essential for our final results and motivates our studies on sequence spaces. For $m\in \Zn$ and $\nu \in \Z$ we define  a $d$-dimensional dyadic cube  %$Q_{\nu, m}$
with sides parallel 
to the axes of coordinates by 
$
Q_{\nu,m} = \prod_{i=1}^d  \left.\left[ \frac{m_i}{2^\nu},\frac{m_i+1}{2^\nu}\right.\right)$, $\nu\in\Z$,  $m=(m_1,\ldots , m_d) \in \Zn$. 
  For $0<u<\infty$, $\nu \in \Z$ and $m\in\Zn$ we denote by $\chi_{\nu, m}^{(u)}$ the $u$-normalised characteristic function of  the cube $Q_{\nu, m}$, 
%\begin{align*}%\label{p-normalized}
 $\ \chi_{\nu, m}^{(u)} = 2^{{\nu d}/{u}}  \chi_{Q_{\nu,m}}$, %= 
%\begin{cases}
%     2^{\frac{\nu d}{u}} & \text{ for } \hspace{0.5cm} x\in Q_{\nu, m}, \\
%     0 &  \text{ for }\hspace{0.5cm} x\notin Q_{\nu, m} ,
%\end{cases}
%\end{align*}
hence $\| \chi_{\nu, m}^{(u)}|L_p\|=1$ and $\| \chi_{\nu, m}^{(u)}|\M\|=1$. 

Let $\widetilde{\phi}$ be a scaling function  on $\R$ with compact support and of sufficiently high regularity.
Let $\widetilde{\psi}$ be an associated wavelet. Then the  tensor-product ansatz yields a scaling function $\phi$  and associated wavelets
$\psi_1, \ldots, \psi_{2^{d}-1}$, all defined now on $\Rn$.  We suppose $\widetilde{\phi} \in C^{N_1}(\R)$ and $\supp \widetilde{\phi}
\subset [-N_2,\, N_2]$ for certain natural numbers $N_1$ and $N_2$. This implies
\begin{equation}\label{2-1-2}
\phi, \, \psi_i \in C^{N_1}(\Rn) \quad \text{and} \quad 
\supp \phi ,\, \supp \psi_i \subset [-N_3,\, N_3]^d , 
\end{equation}
for $i=1, \ldots \, , 2^{d}-1$. We use the standard abbreviations 
\begin{equation}%\label{convention}
\phi_{\nu,m}(x) =  2^{\nu d/2} \, \phi(2^\nu x-m) \quad
\text{and}\quad
\psi_{i,\nu,m}(x) =  2^{\nu d/2} \, \psi_i(2^\nu x-m) . 
\end{equation}

To formulate the result we  introduce some sequence  spaces. For $0<p\le u<\infty$, $0< q \leq\infty$ and  $\sigma\in \R$, let
\begin{multline}
n^{\sigma}_{u,p,q}  :=   \Bigg\{ \lambda = 
\{\lambda_{\nu,m}\}_{\nu,m} : \     \lambda_{\nu,m} \in \C\, , 
\\
\| \, \lambda \, |n^{\sigma}_{u,p,q}\| = \Big\| 
\Big\{2^{\nu(\sigma-\frac{d}{u}) }\,  \Big\|\sum_{m \in
  \Zn}\lambda_{\nu,m}\, \chi^{(u)}_{\nu,m}| \M 
\Big\|\Big\}_{\nu\in\no} | \ell_q\Big\| < \infty \Bigg\}\, . 
\label{mbspqr}
\end{multline}
The following theorem was proved in \cite{Saw2}. 

%%%%%%%%%%%%%%%%%%%%%%%%%%%%%%%%%%%%%%%%%%%%%%%%%%%%%%%%%%%%%%%%%%%%%%%%%%%%%%%%%%%%%%%%%%%%%%%%%%%%%%%%
\begin{theorem}\label{wavemorrey}  
Let $0 < p\le u < \infty$ or $u=p=\infty$, $0<q\le \infty$ and let $s\in \R$. 
Let $\phi$ be a scaling function and let $\psi_i$,  $i=1, \ldots ,2^d-1$, be
the corresponding wavelets satisfying \eqref{2-1-2}. We assume that {$ \max\left\{ (1+\whole{s})_+, \whole{d(\frac 1 p -1)_+ -s} \right\}  \le N_1$}. Then a distribution $f \in \SpRn$ belongs to $\MB(\Rn)$,
if, and only if, 
\begin{align*}
\| \, f \, |\MB(\Rn)\|^\star  = \ &  
\Big\| \left\{\langle f,\phi_{0,m}\rangle \right\}_{m\in \Z^d} |
\ell_u\Big\| 
+ \sum_{i=1}^{2^d-1}
\Big\| \left\{\langle f,\psi_{i,\nu,m}\rangle \right\}_{\nu\in \N_0, m\in \Z^d} | \msib \Big\|
\end{align*}
is finite, where $\sigma=s+\frac d 2 $. Furthermore,
$\| \, f \, |\MB(\Rn) \|^\star $ may be used as an 
equivalent $($quasi-$)$ norm in 
$\MB(\Rn)$.
\end{theorem} 
%%%%%%%%%%%%%%%%%%%%%%%%%%%%%%%%%

\begin{ownremark}\label{wavemorreyrem}
It follows from Theorem~\ref{wavemorrey} that the mapping
\begin{equation}\label{wavemorreyrem1}
T\,:\,f\;\mapsto \; \Big( \left\{\langle f,\phi_{0,m}\rangle \right\}_{m\in \Z^d}, \left\{\langle
f,\psi_{i,\nu,m}\rangle \right\}_{\nu\in \N_0, m\in \Z^d, i=1,\ldots, 2^d-1}\Big)
\end{equation}
is an isomorphism of  $\MB(\Rn) $ onto $\ell_u\oplus
\big(\oplus_{i=1}^{2^d-1} \msib\big)$, $\sigma=s+\frac d 2 $, 
cf. \cite{Saw2}.

The theorem covers the characterisation of Besov spaces $B^s_{p,q}(\Rn)$ by Daubechies wavelets,  cf. \cite[p.26-34]{T-F3} and the references given there. 
\end{ownremark}

\begin{ownremark}
  In \cite{HaSk-bm1} we defined an equivalent norm in the sequence spaces $\msib$ that is more convenient for our purposes. Let $\sigma\in \R$, $0<q\le \infty$, $0<p \le u<\infty$ or $u=p=\infty$.  For a sequence $\{\lambda_{j,m}\}_{j,m}$, $j\in \N_0$, $ m\in \Z^d$, consider a quasi-norm 
\beq
\|\lambda|\msib\|^\ast = \Big(\sum_{j=0}^\infty 2^{qj(\sigma-\frac d u)}\!\!\!\sup_{\nu: \nu \le j; k\in \Z^d}\!\! 2^{qd(j-\nu)(\frac 1 u - \frac 1 p )}\big(\!\sum_{m:Q_{j,m}\subset Q_{\nu,k}}\!\!\!\!|\lambda_{j,m}|^p\big)^{\frac q p}\Big)^{\frac 1 q}\!\!,\,
\label{3-0}
\eeq
with the usual modification if $q=\infty$ or $u=p=\infty$. Then
\begin{eqnarray}	
\msib\, =\, \left\{ \lambda = \{\lambda_{j,m}\}_{j,m}: j\in \N_0,\quad  m\in \Z^d\quad\text{and}\quad  
 \|\lambda|\msib\|^\ast <\infty  \right\}. 
\end{eqnarray}
%In particular, if $0<u = p \le \infty$,  then 
%\begin{eqnarray}	
%\|\lambda|\msib\|^\ast = \, 
%\left(\sum_{j=0}^\infty 2^{qj(\sigma-\frac d p)}\!\Big(\sum_{m\in \Z^d}\!\!|\lambda_{j,m}|^p\Big)^{\frac q p}\right)^{\frac1q}\!\! ,
%\end{eqnarray}
%with the usual modification if $q=\infty$ or $p=\infty$.
\end{ownremark}

\subsection{Function spaces on domains}

%We assume that  $\Omega$ is a bounded $C^{\infty}$ domain in $\Rn$. 
{We assume that  $\Omega$ is a bounded open set  in $\Rn$.}
We consider smoothness Morrey spaces on $\Omega$ defined by restriction.
%Let ${\cal D}(\Omega)$ be the set of all infinitely differentiable functions supported in $\Omega$ and denote by ${\cal D}'(\Omega)$ its dual. 
%  Since we are able to define the extension operator ${\rm ext } : {\cal D}(\Omega) \rightarrow  \SRn$, cf.
%\cite{Saw2010}, the restriction operator ${\rm re } :   \SpRn \rightarrow {\cal D}'(\Omega)$ can be defined naturally as an adjoint operator
%\[
%$\langle {\rm re } (f), \varphi\rangle= \langle f, {\rm ext } (\varphi)\rangle$, $f\in\SpRn$, 
%\]
%where $\varphi\in \mathcal{D}(\Omega)$. We will write $f\vert_{\Omega}={\rm re } (f)$. 

%%%%%%%%%%%%%%%%%%%%%%%%%%%

\begin{definition}\label{D-spaces-Omega}
Let $0 <p\leq  u<\infty$ or $p=u=\infty$, $q\in(0,\infty]$ and  $s\in \R$ (with $u<\infty$ in the case of ${\cal A}={\cal E}$). Then 
$\MA(\Omega)$ is defined by
\[
\MA(\Omega):=\big\{f\in {\cal D}'(\Omega): f=g\vert_{\Omega} \text{ for some } g\in \MA(\Rn)\big\}
\]
endowed with the quasi-norm
\[
\big\|f\mid \MA(\Omega)\big\|:= \inf \big\{ \|g\mid \MA(\Rn)\|:  f=g\vert_{\Omega}, \; g\in  \MA(\Rn)\big\}.
\]
\end{definition}

\begin{ownremark}
  The spaces $\MA(\Omega)$ are quasi-Banach spaces (Banach spaces for $p,q\geq 1$).
  %When $u=p$ we re-obtain the usual Besov and Triebel-Lizorkin spaces defined on bounded smooth domains.   
Several properties of the spaces $\MA(\Omega)$, including  the extension property, were studied in \cite{Saw2010} {in the case when $\Omega$ is  a bounded $C^{\infty}$ domain in $\Rn$.  With the same assumption, embeddings}  %Embeddings 
within spaces in this latter scale as well as to classical spaces like $C(\Omega)$ or $L_r(\Omega)$ were investigated in \cite{hs12b,hs14}. In \cite{hms} we studied the question under what assumptions these spaces consist of regular distributions only.
\end{ownremark}
%%%%%%%%%%%%%%%%%%%%%%%%%%%%%%%%%%%

We recall in detail our compactness result as obtained in \cite{hs12b} (for $\mathcal{A}=\mathcal{N}$) and \cite{hs14} (for $\mathcal{A}=\mathcal{E}$). %{with the assumption  that  $\Omega$ is a bounded $C^{\infty}$ domain.}   
\begin{theorem}  \label{comp}
Let $\Omega$ be a bounded $C^{\infty}$ domain in $\Rn$, $s_i\in \R$, $0<q_i\leq\infty$, $0<p_i\leq u_i<\infty$, $i=1,2$.  Then the embedding 
\begin{equation} \label{bd1comp}
	 \id_{\mathcal{A}}: \MAe(\Omega )\hookrightarrow \MAz(\Omega )
\end{equation}
is compact if, and only if,  
the following condition holds
\begin{equation}\label{cond-comp}%\label{bd3acomp}
\frac{s_1-s_2}{d} > \max\bigg\{0,\frac{1}{u_1} - \frac{1}{u_2}, 
\frac{p_1}{u_1} \Big(\frac{1}{p_1}- \frac{1}{p_2}\Big) \bigg\}.
\end{equation}
\end{theorem}

In case of the target space $L_\infty(\Omega)$ we obtained in \cite{hs12b,hs14} the following result, again for bounded $C^\infty$ domains $\Omega$.

\begin{proposition}\label{MorreyintoLinfty}
 Let $s\in\R$, $0<p < u <\infty$ and $q\in(0,\infty]$. Then
%\begin{equation*}
$\MA(\Omega) \hookrightarrow L_{\infty}(\Omega)$\quad \text{is compact, if, and only if,}\quad 
$s> \frac{d}{u}$.
%\end{equation*}
\end{proposition}

Finally we state some outcome on real interpolation of Besov-Morrey spaces on $\Omega$ which will be of great use for us in the sequel. It can be proved similarly to the proof of \cite[Thm.~1.110]{T-F3},  relying on the corresponding assertions  with $\Omega$ replaced by $\Rn$
that can be found in \cite[Thm.~2.2, Prop.~2.3]{s011a}, 
and on the existence of a common extension operator, cf. \cite[Thm.~5.4]{Saw2010}.

%%%%%%%%%%%%%%%%%%%%%%%%%%%%%%%%%%%
\begin{theorem}\label{ThmInterpolation}
Let  $u, q,q_0,q_1\in(0,\infty]$ and $\theta \in(0,1)$. 
\bli
\item[{\hfill\bf (i)\hfill}]  Let $0< p\le u < \infty$ and  $s_0,s_1\in\R$ with $s_0\neq s_1$. Then  
\[
\MB(\Omega)=\left( {\cal N}^{s_0}_{u,p,q_0}(\Omega),  {\cal N}^{s_1}_{u,p,q_1}(\Omega)  \right)_{\theta,q}  \qquad \text{if} \qquad   s=(1-\theta)s_0+\theta s_1. 
\]
\item[{\hfill\bf (ii)\hfill}] Let $1\le p\le u < \infty$ and $s\in\R$. Then 
\[
\MB(\Omega)=\left( {\cal N}^{s}_{u,p,q_0}(\Omega),  {\cal N}^{s}_{u,p,q_1}(\Omega)  \right)_{\theta,q} \qquad \text{if} \qquad \frac 1q=\frac{1-\theta}{q_0}+\frac{\theta}{q_1}.
\]
 \eli
\end{theorem}

\subsection{Entropy numbers}\label{sect-ek}
As explained in the beginning already, our main concern in this paper is to characterise the compactness of embeddings in further detail. Therefore we briefly recall the concept of entropy numbers.

\begin{definition}\label{defi-ak}
Let $ X $ and $Y$ be two complex (quasi-) Banach spaces, 
$k\in\nat\ $ and let $\ T\in\mathcal{L}(X,Y)$ be a linear and 
continuous operator from $ X $ into $Y$. 
The {\em k\,th entropy number} $\ e_k(T)\ $ of $\ T\ $ is the
infimum of all numbers $\ \varepsilon>0\ $ such that there exist $\ 2^{k-1}\ $ balls
in $\ Y\ $ of radius $\ \varepsilon\ $ which cover the image $\ T\,B_X$ of the unit ball $\ B_X=\{x\in
X:\;\|x|X\|\leq 1\}$.
\end{definition}

\begin{ownremark}
For details and properties of entropy numbers we refer to \cite{CS,EE,Koe,Pie-s} (restricted to the case of Banach spaces), and \cite{ET} for some extensions to quasi-Banach spaces. Among other
features we only want to mention 
the multiplicativity of entropy numbers: let $X,Y,Z$
be complex (quasi-) Banach spaces and $\ T_1 \in\mathcal{L}(X,Y)$, $ T_2 \in\mathcal{L}(Y,Z)$. Then
\beq
e_{k_1+k_2-1} (T_2\circ T_1) \leq e_{k_1}(T_1)\,
e_{k_2} (T_2),\quad k_1, k_2\in\nat.
\label{e-multi}
\eeq

Note that one has  
$\ \lim_{k\rightarrow\infty} e_k(T)= 0\ $ {if, and only if,} 
$\ T\ $ {is compact}. 
The last equivalence justifies the saying that entropy numbers measure
`how compact' an operator acts. This is one reason to study the asymptotic
behaviour of entropy numbers (that is, their decay) for compact operators in
detail. Another one is the application to estimate eigenvalues; we refer to the monographs \cite{CS,EE,ET,Koe,Pie-s} for further details.

\ignore{
Another very prominent concept is that of approximation numbers  $a_k(T)$  of an operator $T \in\mathcal{L}(X,Y)$, defined by 
$$a_k(T) = \inf \{ \| T - S \| : S \in {\mathcal L}(X,Y), \,
\mathrm{rank}\, S < k \},\quad k\in\nat.$$
They can -- unlike entropy numbers -- be regarded as special {\itshape
  $s$-numbers}, a concept introduced by {Pietsch} \cite[Sect.~11]{Pia}. Of special importance is the close connection of both concepts, entropy numbers as well as approximation numbers, with spectral theory, in particular, the estimate of eigenvalues. We refer to the monographs \cite{CS,EE,ET,Koe,Pie-s} for further details.}
\end{ownremark}

We recall the following property of entropy numbers which was obtained (in the Banach case situation) in \cite[Sect.~1.16.2]{T-I} and \cite[12.1]{Pia} and extended to the quasi-Banach case setting in  \cite[Theorem~3.2]{HT1} and \cite[Theorem~1.3.2]{ET}. Let $A$ be a
quasi-Banach space and let $\{ B_0 , B_1 \}$ be an interpolation couple of
quasi-Banach spaces. Let $0 < \theta <1$ and let $B_\theta$ be a quasi-Banach
space such that
%\[
$B_0 \cap B_1 \hra B_\theta \hra B_0 + B_1 $  \mbox{(naturally quasi-normed)}
%\]
and
\begin{equation}      \label{5.25}
\| b \, | B_\theta \| \le \| b \, | B_0 \|^{1- \theta} \, \| b \, | B_1
\|^\theta \quad \mbox{for all} \quad b \in B_0 \cap B_1.
\end{equation}
Let $T \in \mathcal{L}(A, B_0 \cap B_1).$ Then there is a number $c>0$ such that for all $k \in \nat$,
\begin{equation}      \label{ek_ipol}
e_{2k}(T  : A \hra B_\theta ) \le c \, e^{1- \theta}_k (T : A \hra B_0 ) \, e^\theta_k (T : A \hra B_1 ).
\end{equation}

%%%%%%%%%%%%%%%%%%%%%%%%%%%%%%%%%%%%%%%%%%%%%
Now we benefit from the technique of quasi-normed operator ideals. In the context of entropy numbers the approach goes back to Carl \cite{Carl81}. It was used for Sobolev embeddings first time in \cite{kuehn-exp,K-L-S-S}. For a bounded linear operator $T \in \cl (X,Y)$, where
$X$ and $Y$ are (quasi)-Banach spaces,  and a positive real number $r$ we put 
\begin{equation}\label{3-2-14}
L_{r,\infty}^{(e)} (P) := \sup_{n \in \N} \, n^{1/r} \, e_n (P) \, .
\end{equation}
This is a quasi-norm (in general not a norm)
for the operator ideal of all operators $P$ with
$L_{r,\infty}^{(e)} (P) < \infty$, cf. Pietsch \cite{Pia,Pie-s}. 

One of the main tools in our arguments 
will be the characterisation of the asymptotic 
behaviour of 
the entropy numbers of the embedding $\ell_{p_1}^N 
\hookrightarrow \ell_{p_2}^N$. For all $n \in \N$ we have
\begin{equation}\label{Schuett}
e_k \Big(\id: \, \ell_{p_1}^N \hra \ell_{p_2}^N\Big) \sim
\begin{cases}
1 & \mbox{if}\quad 1 \le k \le \log 2N , \\
\Big( \frac{\log (1+\frac{N}{k})}{k}\Big)^{\frac{1}{p_1}-\frac{1}{p_2}} 
& \mbox{if}\quad \log 2N \le k \le 2N, \\
2^{-\frac{k}{2N}} \, N^{\frac{1}{p_2}-\frac{1}{p_1}}
& \mbox{if}\quad 2N \le k, 
\end{cases}
\end{equation}
and if $0<p_2<p_1\le \infty $ it holds 
\begin{equation}\label{Pietsch}  
e_k \Big(\id: \, \ell_{p_1}^N \hra \ell_{p_2}^N\Big) \sim 2^{-\frac{k}{2N}} \, N^{\frac{1}{p_2}-\frac{1}{p_1}} \qquad \text{for all}\quad k\in \N. 
\end{equation}
In the case $1 \le p_1,p_2 \le \infty$ this has been proved by
Sch\"utt \cite{Sch}. For $p_1<1$ and/or $p_2<1$  
we refer to Edmunds and Triebel \cite{ET} and
Triebel \cite[7.2,\, 7.3]{TrFS} (with a little supplement in \cite{Kue2}).

The characterisation of the asymptotic behaviour of the entropy numbers $e_k\big( \ell^N_{p_1} \hookrightarrow \ell^N_{p_2} \big)$, recall \eqref{Schuett} and \eqref{Pietsch}, implies
\begin{equation}\label{3-2-16}
\lrie (\id : \ell_{p_1}^{N} \hra \ell_{p_2}^{N})  
\sim   
N^{\frac 1r - \frac{1}{p}} \quad  \mbox{if}\quad \frac 1 r > \max \Big(0, \frac{1}{p}\Big) \ \text{and} %where}
\ \frac1p=\frac{1}{p_1}-\frac{1}{p_2}.
\end{equation}

\section{Entropy numbers of compact embeddings in sequence spaces}\label{entro_seq}

Recall our remarks about the wavelet characterisation of spaces $\MB(\Rn)$, in particular, Theorem~\ref{wavemorrey}. Our strategy in estimating corresponding entropy numbers will be to transfer the question to the appropriate sequence spaces via wavelet decompositions. This method will be explained in further detail in the beginning of Section~\ref{entro_func} below. We begin with the sequence space problem and have thus to adapt our sequence spaces to the spaces on bounded domains first.

\begin{ownremark}
Note that we use the most convenient setting for this purpose: we assume (implicitly) that the supports of the corresponding distributions are inside the domain, thus avoiding boundary wavelets. It turns out that this apparently simpler situation is already sufficient for our problem. This justifies our approach below. Note that we already used a similar argument in \cite{hs12b,hs14}  when dealing with the compactness assertion.
\end{ownremark}

Let $Q$ be a unit cube, $0<p\leq u<\infty$, $\sigma\in\R$, $0<q\leq\infty$.  We define a sequence space $\mbt(Q)$ putting 
\[
\mbt(Q) :=   \Big\{ \lambda = 
\{\lambda_{j,m}\}_{j,m} : \lambda_{j,m} \in \C ,\   Q_{j,m}\subset Q,\ \text{and} \ \|  \lambda \, |\mbt\|<\infty\Big\},
\]
where
\[\|  \lambda \, |\mbt\| = \left(\sum_{j=0}^\infty 2^{jq(\sigma-\frac{d}{u}) }  
\!\!\!\!\sup_{\nu: \nu \le j; k: Q_{\nu,k}\subset Q}\! 2^{qd(j-\nu)(\frac 1 u - \frac 1 p )}\Big(\!\sum_{m:Q_{j,m}\subset Q_{\nu,k}}\!\!\!|\lambda_{j,m}|^p\Big)^{\frac q p}\right)^{\frac1q},
\]
with the usual modification when $q=\infty$. Moreover for fixed $j\in \N_0$ we put  
\begin{align*}
\mmb = \{ \lambda = 
\{\lambda_{j,m}\}_{j,m} : \quad     \lambda_{j,m} \in \C \quad  Q_{j,m}\subset Q \quad \text{and}\quad\|\lambda|\mmb\|< \infty\},
\end{align*}
{where} \quad
$\ds \|\lambda|\mmb\| = \sup_{\nu: \nu \le j; k\in \Z^d}\! 2^{d(j-\nu)(\frac 1 u - \frac 1 p )}\Big(\sum_{m:Q_{j,m}\subset Q_{\nu,k}\subset Q}\!\!|\lambda_{j,m}|^p\Big)^{\frac 1 p}$.

\begin{ownremark}\label{seq-n-b}
If $p=u$, then the spaces $\tilde{n}^\sigma_{p,p,q}(Q)=\tilde{b}^\sigma_{p,q}(Q)$ coincide with the spaces $\ell_q(2^{j(\sigma-\frac{d}{p})}\ell_p^{M_j})$, $M_j\sim 2^{jd}$, as dealt with in \cite[Sect.~8.1]{TrFS}. Then in case of the (sequence space version of the) compact embedding \eqref{bd1comp} recalled in Theorem~\ref{comp}, that is,
%\[
$\ \id: \mbet \hookrightarrow \mbzt$,
%\]
it can always be extended to the chain of embeddings
%\[
$\tilde{b}^{\sigma_1}_{u_1,q_1} \hookrightarrow \mbet \hookrightarrow \mbzt \hookrightarrow \tilde{b}^{\sigma_2}_{p_2,q_2}$,
%\]
where the first embedding is just by monotonicity and the last using H\"older's inequality and the boundedness of $Q$. Thus the multiplicativity of entropy numbers \eqref{e-multi} immediately leads to
\beq\label{ek_below_class}
e_k\left(\mbet \hookrightarrow \mbzt\right) \geq \ c\ e_k\left(\tilde{b}^{\sigma_1}_{u_1,q_1} \hookrightarrow \tilde{b}^{\sigma_2}_{p_2,q_2}\right)\sim \ k^{-\frac{\sigma_1-\sigma_2}{d}}
\eeq
in all cases admitted by \eqref{cond-comp}; the latter equivalence can be found in \cite[Thm.~8.2]{TrFS}.
\end{ownremark}

\begin{ownremark}
  It is obvious from Theorem~\ref{wavemorrey} that the smoothness parameters $s$ and $\sigma$, appearing on the function spaces and sequence spaces side, respectively, are linked by $\sigma=  s+\frac{d}{2}$, so in all cases below where the difference $\sigma_1-\sigma_2$ appears, we could equally call it $s_1-s_2$. 
  % However, we have decided that in the present section dealing with the sequence space results only, we shall always stick to
  %the smoothness parameters $\sigma$ or $\sigma_i$, $i=1,2$.  
\end{ownremark}

We collect some recent results from our paper \cite{HaSk-sm} needed for our arguments below.

\begin{lemma}[cf. \cite{HaSk-sm}]\label{lemma15030}
Let $0< p_i\le u_i<\infty$, $i=1,2$, and $j\in \no$ 
be given. Then the norm of the compact identity operator  
\begin{equation}\label{id_j-m}
 \id_j: \mmbet\hookrightarrow \mmbzt
\end{equation}
satisfies
\begin{equation}\label{1503-0}
\|\id_j\| = 
\begin{cases}
1 &\qquad \text{if} \quad p_1\ge p_2\quad \text{and }\quad u_2\ge u_1,\quad  \\
 1 &\qquad \text{if} \quad p_1 < p_2\quad \text{and }\quad \frac{p_2}{u_2} \le \frac{p_1}{u_1},\\  
 2^{jd(\frac{1}{u_2}-\frac{1}{u_1})}&\qquad \text{if} \quad p_1\ge p_2\quad \text{and }\quad u_2 < u_1,\\
 \end{cases}
\end{equation}
and in the remaining case, there is a constant $c$, $0<c\le 1$, independent of $j$ such that 
\begin{equation}\label{1503-a}
c\, 2^{jd(\frac{1}{u_2}-\frac{p_1}{u_1p_2})} \le\|\id_j\| \leq 
2^{jd(\frac{1}{u_2}-\frac{p_1}{u_1p_2})}  \qquad \text{if} \quad  
p_1 < p_2 \ \text{and } \ \frac{p_2}{u_2} > \frac{p_1}{u_1}\ .
\end{equation}
\end{lemma}

Using Sch\"utt's result \eqref{Schuett} and \eqref{Pietsch} together with some embedding arguments we obtained in \cite{HaSk-sm} first entropy number estimates.

\begin{lemma}[cf. \cite{HaSk-sm}]\label{lemma1503}
Let $j\in\nat$, $0<p_i\leq u_i<\infty$, $i=1,2$, and $k\in \no$ with $k\gtrsim 2^{jd}$. 
Then
\begin{equation}\label{1503-1}
e_k(\id_j:\mmbet \rightarrow \mmbzt) \sim  2^{-k2^{-jd}} \ 2^{jd\left(\frac{1}{u_2}-\frac{1}{u_1}\right)}.
\end{equation}
\end{lemma}

Now we are able to extend this in some sense. Recall that for a bounded subset $K$ of a finite-dimensional (quasi)-Banach space $Y$, the $k$-th entropy number $e_k(K,Y)$ is defined as 
\begin{equation*}
e_k(K,Y) = \min\{\varepsilon>0 : \; K\subset \bigcup_{i=1}^{2^{k-1}} (x_i+\varepsilon B_Y))\; \text{for some} \; x_1, \ldots x_{2^{k-1}}\in Y\}\, .
\end{equation*}  
If $X$ denotes the vector space $Y$ equipped with another (quasi)-norm, then 
%\begin{equation*}
$e_k(B_X,Y) \,=\, e_k(\id: X\rightarrow Y)$.
%\end{equation*}
One can also easily check that 
%\begin{equation*}
$e_k(rK,Y) = r e_k(K,Y)$
%\end{equation*}
if $r>0$.

\begin{lemma}\label{l:u2p1}
Let $j\in\nat$, $0<p_2\leq u_2<p_1\le u_1<\infty$  and $k\in \no$. 
Then
\begin{equation}\label{1503-2}
e_k(\id_j:\mmbet \rightarrow \mmbzt) \sim  2^{-k2^{-jd}} \ 2^{jd\left(\frac{1}{u_2}-\frac{1}{u_1}\right)}.
\end{equation}
\end{lemma}

\begin{proof}
Let $B_q(0,r)$, $r>0$, denote the ball of radius $r$ centred at the origin in the space $\ell_q^{2^{jd}}$ and $B_{u,p}(0,r)$ the corresponding ball in the space $\mmb$. 
Directly from the definition of the norm in the space $\mmbet$ it follows that
\begin{equation}\label{balls}
B_{u}(0,1) \subset B_{u,p}(0,1)\subset B_{p}(0,2^{jd(\frac{1}{p}-\frac{1}{u})}).
\end{equation} 
In consequence,  if $q< p_1\le u_1$, then 
\begin{align}\label{ebb}
e_k(B_{u_1,p_1}(0,1),\ell_q) \ge & \ e_k(B_{u_1}(0,1),\ell_q^{2^{jd}}) \sim 2^{-k2^{-jd}} \ 2^{jd\left(\frac{1}{q}-\frac{1}{u_1}\right)}\\
\intertext{and} 
e_k\big(B_{u_1,p_1}(0,1),\ell_q^{2^{jd}}\big) \le & \ e_k\big(B_{p_1}(0,2^{jd(\frac{1}{p_1}-\frac{1}{u_1})}),\ell_q^{2^{jd}}\big) \nonumber\\
=  & \   
2^{jd(\frac{1}{p_1}-\frac{1}{u_1})} e_k\big(B_{p_1}(0,1),\ell_q^{2^{jd}}\big) \nonumber\\
\sim &\ 2^{-k2^{-jd}}  2^{jd\left(\frac{1}{q}-\frac{1}{u_1}\right)}.\label{eba}
\end{align}
Now, taking $q=u_2$ in \eqref{eba}, we obtain
\begin{align}
  e_k(\id_j:\mmbet \rightarrow \mmbzt)
  & \le e_k(\id_j:\mmbet \rightarrow \ell_{u_2}^{2^{jd}}) \|\id: \ell_{u_2}^{2^{jd}} \rightarrow 
 \mmbzt\| \nonumber\\
 &\sim 2^{-k2^{-jd}}  2^{jd\left(\frac{1}{u_2}-\frac{1}{u_1}\right)}  \label{eba1}
\end{align} 
since $\|\id: \ell_{u_2}^{2^{jd}} \rightarrow  \mmbzt\|=1$, cf. Lemma~\ref{lemma15030}. 
On the other hand, taking $q=p_2$ in \eqref{ebb}, we get 
\begin{align}
  c 2^{-k2^{-jd}}  2^{jd\left(\frac{1}{p_2}-\frac{1}{u_1}\right)} & \le e_k(\id_j:\mmbet \rightarrow \ell_{p_2}^{2^{jd}}) \nonumber\\
\nonumber  & \le e_k(\id_j:\mmbet \rightarrow \mmbzt) \|\id: \mmbzt  \rightarrow \ell_{p_2}^{2^{jd}}\| \\ & \sim  2^{jd\left(\frac{1}{p_2}-\frac{1}{u_2}\right)} e_k(\id_j:\mmbet \rightarrow \mmbzt) , \label{ebb1}
\end{align} 
 in view of $\|\id:\mmbzt  \rightarrow \ell_{p_2}^{2^{jd}} \|=  2^{jd\left(\frac{1}{p_2}-\frac{1}{u_2}\right)}$, cf. Lemma~\ref{lemma15030}. This proves \eqref{1503-2}. 
\end{proof}

The next result is well-known in the Banach case,  cf.  \cite{Carl81}, \cite[p.63]{pisier} or 
\cite[Prop.~II.1.3]{EE}. The extension to the  quasi-Banach case is standard.
%For convenience we give a short proof extended to the situation of quasi-Banach spaces here.

\begin{lemma}\label{dimN}
Let $(E, \|\cdot\|)$ be an $N$-dimensional complex quasi-Banach space. Then %there is a constant $C>0$ independent of $N$ such that 
\[
e_k(\id: E\rightarrow E) \sim % \le C 
\begin{cases} 1 & \text{if}\qquad k\le 2N, \\
2^{-\frac{k-1}{2N}} & \text{if}\qquad k > 2N. 
\end{cases}
\] 
\end{lemma}

\ignore{
\begin{proof}
Let $B(0,r)$ denote a ball of radius $r$ in $E$ and  $K$  the constant  in the quasi-triangle inequality. 
Let $\mu$ be the Lebesgue measure on $E=\R^{2N}$ such that $\mu\big(B(0,1)\big)=1$. For given $\varepsilon$, $0<\varepsilon\le 1$, let $I_\varepsilon$ be a maximal family of points $x_1,\ldots , x_m$ in  $B(0,1)$ such that  $\|x_i-x_j\|\ge K\varepsilon$ if $i\not= j$. Then the balls $x_i+B(0,\varepsilon/2)$ are pairwise disjoint.  
Moreover any ball $x_i+B(0,\varepsilon/2)$ is contained in $B(0,2K)$.  So by the volume argument we get 
\[ m \cdot \big(\varepsilon/2\big)^{2N} \le \big(2K\big)^{2N}. 
\]
For sufficiently large $k$, $k\ge 2N + \log 4K$,  we choose $\varepsilon = 4K 2^{-(k-1)/2N}$.  
We get $m\le 2^{k-1}$. So $2^{k-1}$ balls of radius $K\varepsilon$ cover the unit ball. If $k$ is small, then $e_k \le 1 =\|\id\| $. This proves the lemma.  
\end{proof}
}
This leads to the following operator ideal estimate.

\begin{lemma}\label{idealmm}
Let $0< p_i\le u_i<\infty$, $i=1,2$, and $j\in \no$ be given. Assume that 
$p_1\ge p_2$,  or $p_1 < p_2$  and $\frac{p_2}{u_2}\le \frac{p_1}{u_1}$. Let $\frac{1}{u}=\frac{1}{u_1}-\frac{1}{u_2}$. Then
\begin{equation} \label{idelmm1}
\lrie (\id : \,\mmbet \rightarrow \mmbzt) \sim 
2^{jd(\frac 1r - \frac{1}{u})} \quad  \mbox{if}\quad \frac 1 r > \max \Big(0, \frac{1}{u}\Big).
\end{equation}  
\end{lemma}

\begin{proof}
If $u_2<p_1$, then \eqref{idelmm1} follows from Lemma~\ref{l:u2p1}. 

Let $p_1\le u_2$ and let $e_k=e_k(\id_j:\mmbet \rightarrow \mmbzt)$. Lemma~\ref{lemma1503} implies 
\begin{equation}\label{klarge}
\sup_{k\ge c2^{jd}} k^\frac{1}{r} e_k \;\sim\;  2^{jd(\frac 1r - \frac{1}{u})} .
\end{equation}  
Thus to verify \eqref{idelmm1} it is sufficient to prove the estimates from above for $1\le k\le c2^{jd}$. 
If $\log c2^{jd}\le k \le c2^{jd}$, then by the Sch\"utt estimates \eqref{Schuett} we get 
\begin{align*}
  k^\frac{1}{r} e_k  &\le  k^\frac{1}{r} e_k(\id_j:\mmbet \rightarrow \ell_{u_2}^{2^{jd}}) \le  k^\frac{1}{r} 2^{jd(\frac{1}{p_1}-\frac{1}{u_1})} e_k\big(B_{p_1}(0,1),\ell_{u_2}\big) 
  \\
& \le k^\frac{1}{r} 2^{jd(\frac{1}{u_2}-\frac{1}{u_1})} \big(\frac{k}{2^{jd}}\big)^{\frac{1}{u_2}-\frac{1}{p_1}} \big(\log\big(1+\frac{2^{jd}}{k}\big)\big)^{\frac{1}{p_1}-\frac{1}{u_2}} \le c 2^{jd(\frac{1}{r} +  \frac{1}{u_2}-\frac{1}{u_1})}.  \nonumber
\end{align*} 
If  $1\le k \le \log c2^{jd}$, then Lemma~\ref{dimN} implies 
\begin{align*}
k^\frac{1}{r} e_k &\le k^\frac{1}{r} e_k(\id: \mmbet \rightarrow \mmbet)\, \|\id: \mmbet \rightarrow \mmbzt \| \le k^\frac{1}{r} 
\max(1,2^{-jd\frac{1}{u}}) \\
& \le c 2^{jd(\frac{1}{r}-\frac{1}{u})},
\end{align*}
since $k^{1/r}\le c 2^{jd(\frac{1}{r}-\frac{1}{u})}$ if $\frac{1}{r}>  \frac{1}{u}>0$ and $1<k\le \log c2^{jd}$. 
This proves \eqref{idelmm1}.
\end{proof}

\begin{proposition}\label{slarge}
	Let  $\sigma_i\in \R$, $0<q_i\leq\infty$, $0<p_i\leq u_i<\infty$, $i=1,2$.  Assume $p_1\le u_2$ and  
	\begin{equation}\label{bd3acomplarge}
	\frac{\sigma_1-\sigma_2}{d} >  \frac{1}{p_1}- \frac{1}{u_2} .
	\end{equation}
Then there is some $c>0$ such that 
\begin{equation}\label{1603-1large}
  e_k \big(\mbet\hookrightarrow \mbzt\big) 
  \le c   k^{-\frac{\sigma_1-\sigma_2}{d}} .  
\end{equation}	
\end{proposition}

\begin{proof}
	Let $M\in \N$.  We decompose the operator $\id: \mbet\rightarrow \mbzt$ into the sum of two operators 
\begin{equation}\label{P_M}
P_M= \sum_{j=0}^M \id_j ,\quad Q_M= \sum_{j=0}^M \id_j ,  \quad\text{where}\quad
%	where 
%\[
(\id_j \lambda)_{\nu,\ell} := 
	\begin{cases} 
	\lambda_{\nu,\ell} & \text{if}\ \nu=j, \\
	0 & \text{if}\ \nu\not= j . 
	\end{cases}
        \end{equation}%\]
	Then by \eqref{1503-0} we have 
	\begin{align*}
	e_k(\id_j:\mbet\rightarrow \mbzt)  \le  & 2^{-j\delta} e_k \big(\mmbet\hookrightarrow \mmbzt\big) \\   \le& 2^{-j\delta} \|\mmbet\hookrightarrow\ell^{2^{jd}}_{p_1}\| e_k(\ell^{2^{jd}}_{p_1} \rightarrow \ell^{2^{jd}}_{u_2}) \|\ell^{2^{jd}}_{u_2}\hookrightarrow \mmbzt\|  \nonumber\\ \le &
	C 2^{-j\delta}  2^{jd(\frac{1}{p_1}-\frac{1}{u_1})}e_k(\ell^{2^{jd}}_{p_1} \rightarrow \ell^{2^{jd}}_{u_2}), \nonumber
	\end{align*} 
	where $\delta=\sigma_1-\sigma_2-\frac{d}{u_1}+\frac{d}{u_2}$. In consequence,
	\begin{align}
	  \lrie(\id_j) & \le\ C   2^{-j\delta}  2^{jd(\frac{1}{p_1}-\frac{1}{u_1})} 2^{jd(\frac{1}{r}+\frac{1}{u_2}-\frac{1}{p_1})}\nonumber\\
          &= \ C 2^{jd(\frac{1}{r} - \frac{\sigma_1-\sigma_2}{d})},   \quad \text{if}\quad\frac{1}{r} > \max \big( 0, \frac{1}{p_1}-\frac{1}{u_2}\big),   
\label{1603-2l}	\end{align}
	cf. \eqref{3-2-16}. 
	We recall, that for any $r>0$ there exist positive numbers $\rho\le 1$ and $C>0$ such that 
	\begin{equation}\label{3-2-19l}
	\lrie (\sum_j T_j)^\varrho \le \ C\ \sum_{j}  \lrie (T_{j})^\varrho\, , 
	\end{equation} 
	cf. Pietsch \cite[6.2]{Pia} and K\"onig, \cite[1.c.5]{Koe}.
	Hence, \eqref{1603-2l} and \eqref{3-2-19l} yield
	\begin{align*} %\label{3-2-20l}
	\lrie (P_M)^\varrho  \le  \sum_{j=0}^M \lrie (\id_{j})^\varrho 
	\le  c\, 
	\sum_{j=0}^M  2^{\varrho jd(\frac{1}{r} - \frac{\sigma_1-\sigma_2}{d})} \le c' 2^{\varrho Md(\frac{1}{r} - \frac{\sigma_1-\sigma_2}{d})} ,
	\end{align*}
	if we choose $r$ such that  
%	\[
$\	\frac{1}{r} - \frac{\sigma_1-\sigma_2}{d}>0 \ $ and $\ \frac{1}{r} > \frac{1}{p_1}-\frac{1}{u_2}$. 
%	\]
	Thus 
	\begin{equation}\label{1603-3l}
	e_{2^{Md}} (P_M) \le C 2^{-Md(\frac{\sigma_1-\sigma_2}{d})} .
	\end{equation}
	Similarly, 
	\begin{align*} %\label{3-2-21l}
	\lrie (Q_M)^\varrho  \le  \sum_{j=M+1}^\infty \lrie (\id_{j})^\varrho 
	\le  c\, 
	\sum_{j=M+1}^\infty  2^{\varrho jd(\frac{1}{r} - \frac{\sigma_1-\sigma_2}{d})} \le c' 2^{\varrho Md(\frac{1}{r} - \frac{\sigma_1-\sigma_2}{d})} ,
	\end{align*}
	if we choose $r$ such that  
%	\[
$\	0 \le  \frac{1}{p_1}-\frac{1}{u_2} <\frac{1}{r}< \frac{\sigma_1-\sigma_2}{d} $. 
%	\]
	Thus 
	\begin{equation}\label{1603-4l}
	e_{2^{Md}} (Q_M) \le C 2^{-Md(\frac{\sigma_1-\sigma_2}{d})} .
	\end{equation} 
	Now by \eqref{1603-3l} and \eqref{1603-4l} we can conclude that 
	\begin{equation}\label{1603-5}
	e_{2^{Md}} (\id: \mbet\rightarrow \mbzt) \le C 2^{-Md(\frac{\sigma_1-\sigma_2}{d})} 
	\end{equation} 
	and by standard arguments
%	\begin{equation}\label{1603-5l}
	$\ e_{k} (\id: \mbet\rightarrow \mbzt) \le C k^{-\frac{\sigma_1-\sigma_2}{d}}$, $k\in\nat$.
        
%	\end{equation} 
	\end{proof}

Now we are ready to establish our first main result in this context which reflects the `classical' situation as we shall see below.

\begin{theorem}\label{ek_classical}
	Let  $\sigma_i\in \R$, $0<q_i\leq\infty$, $0<p_i\leq u_i<\infty$, $i=1,2$. Assume that 
\bli
\item[{\hfill\bf (i)\hfill}] either
	\begin{equation}\label{bd3acomp-0}
	\frac{ \sigma_1-\sigma_2}{d} > \max\bigg\{0,\frac{1}{u_1} - \frac{1}{u_2}, 
	\frac{p_1}{u_1} \Big(\frac{1}{p_1}- \frac{1}{p_2}\Big) \bigg\}= 0,
	\end{equation}
\item[{\hfill\bf (ii)\hfill}] or
\begin{equation}\label{bd3acompu}
\frac{\sigma_1-\sigma_2}{d} > \max\bigg\{0,\frac{1}{u_1} - \frac{1}{u_2}, 
\frac{p_1}{u_1} \Big(\frac{1}{p_1}- \frac{1}{p_2}\Big) \bigg\} = \frac{1}{u_1} - \frac{1}{u_2}>0.
\end{equation}
Then
\begin{equation}\label{1603-1n}
  e_k \big(\mbet\hookrightarrow \mbzt\big) \sim 
  k^{-\frac{\sigma_1-\sigma_2}{d}} .  
\end{equation}
\eli
\end{theorem}

\begin{proof}
{\em Step 1.}\quad We prove (i). 
It follows from \eqref{bd3acomp-0} that $u_2\le u_1$  and $p_2\le p_1$. So we may consider two cases: $p_2\le u_2 < p_1\le u_2$ or $p_2\le p_1\le u_2\le u_1$.   

\emph{Substep 1.1}.	In the first case  Lemma~\ref{l:u2p1} implies  that 
	\begin{eqnarray}
	\lrie (\id:\mmbet\hookrightarrow \mmbzt ) \sim 2^{jd(\frac{1}{r}-\frac{1}{u_1}+\frac{1}{u_2})}, \qquad \frac{1}{r} > 0 . 
	\end{eqnarray}
	In consequence, 
	%\begin{eqnarray}
	$\ \lrie (\id_j ) \sim 2^{jd(\frac{1}{r}-\frac{\sigma_1-\sigma_2}{d})}$
%	\end{eqnarray}
	for any $\frac{1}{r} > 0 $. Now we can prove the estimate from above in the same way as in Proposition~\ref{slarge}. The estimate from below follows from Lemma~\ref{l:u2p1} and the obvious inequality
	\[
     	e_{2^{jd}} \big(\mmbet\hookrightarrow \mmbzt\big)\le 2^{j\delta} e_{2^{jd}} \big(\mbet\hookrightarrow \mbzt\big) .
	\]

        \emph{Substep 1.2}. Let now $p_2\le p_1\le u_2\le u_1$. We consider the diagram
        
\begin{equation}
\begin{tikzcd}
\mmbet \arrow[rd, "\Id_1"] \arrow[r, "\id"] & \mmbzt \\
& \mmbet \arrow[u, "\Id_2"]
\end{tikzcd}
\end{equation}
By Lemma~\ref{lemma15030} we have $\|\id_2\|= 2^{jd(\frac{1}{u_2}-\frac{1}{u_1})}$. Moreover Lemma~\ref{dimN} implies 
$\lrie(\Id_1) \le c 2^{\frac{jd}{r}}$. 
In consequence,
\begin{equation}\label{3103-1}
\lrie(\id_j) \le 2^{-j\delta} 2^{jd(\frac{1}{u_2}-\frac{1}{u_1})} \lrie(\Id_1) \le 2^{jd(\frac{1}{r}-\frac{\sigma_1-\sigma_2}{d})}  . 
\end{equation} 
In the same way as above we show that 
%\begin{equation}%\label{3103-2}
$\ e_{2^{Md}} (P_M) \le C 2^{-Md(\frac{\sigma_1-\sigma_2}{d})} \ $ {and} $\ e_{2^{Md}} (Q_M) \le C 2^{-Md(\frac{\sigma_1-\sigma_2}{d})}$ . 
%\end{equation} 
The estimate from below follows from Lemma~\ref{lemma1503}. \\
{\em Step 2.}\quad We prove (ii). 
Let $p_1\ge p_2$  and $u_1< u_2$, or  $p_1< p_2$ and $\frac{p_2}{u_2} \le \frac{p_1}{u_1}$. In both cases we can use Lemma~\ref{idealmm}. So the estimate from above can be proved in the same way as in Proposition~\ref{slarge}. The estimate from below follows once more from Lemma~\ref{lemma1503}. 
\end{proof}

%%%%%%%%%%%%%%%%
%%%%%%%%%%%%%%%%
%%%%%%%%%%%%%%%%
%%%%%%%%%%%%%%%%

\begin{ownremark}
We call this the `classical' setting, as in case of $p=u$ the sequence spaces $\tilde{n}^\sigma_{p,p,q}=\tilde{b}^\sigma_{p,q}$ coincide, recall Remark~\ref{seq-n-b}. As mentioned in \eqref{ek_below_class}, it is well-known that
%\beq
$\ e_k(\tilde{b}^{\sigma_1}_{u_1,q_1} \hookrightarrow \tilde{b}^{\sigma_2}_{p_2,q_2})\sim \ k^{-\frac{\sigma_1-\sigma_2}{d}} $
%\eeq
in all cases admitted by \eqref{cond-comp}; see \cite[Thm.~8.2]{TrFS}. This behaviour is now extended to spaces $\mbt$ for all $\ 0<p\leq u<\infty$ whenever \eqref{bd3acomp-0} or \eqref{bd3acompu} are satisfied.
\end{ownremark}

In view of our compactness result Theorem~\ref{comp}, in particular, \eqref{cond-comp}, and Theorem~\ref{ek_classical}, it remains to deal with the case when
\[
\max\Big\{0,\frac{1}{u_1} - \frac{1}{u_2}\Big\} 
< \frac{p_1}{u_1} \Big(\frac{1}{p_1}- \frac{1}{p_2}\Big)  
\]
which is equivalent to  $p_1 < p_2$   and $\frac{p_2}{u_2} > \frac{p_1}{u_1}$. This excludes, in particular, the setting $p_1=u_1$. We first give the counterpart of Proposition~\ref{slarge} in this case.

%%%%%%%%%%%%%%%%%%%%%%%%%%%%%%%%%%%%%%%%%%%%%%%%%%%%%%%%%%%%%%%%%%%%%%%%%%%%%%%%%%
%%%%%%%%%%%%%%%%%%%%%%%%%%%%%%%%%%%%%%%%%%%%%%%%%%%%%%%%%%%
%%%%%%%%%%%%%%%%%%%%%%%%%%%%%%%%%%%%%%%%%%%%%%%%%%%%%%%%%%%%%%%%%%%%%%%%%%%%
\begin{proposition}\label{slargebis}
	Let  $\sigma_i\in \R$, $0<q_i\leq\infty$, $0<p_i\leq u_i<\infty$, $i=1,2$.  Assume that
$p_1<p_2$, $\frac{p_1}{u_1}< \frac{p_2}{u_2}$ and  
	\begin{equation}\label{bd3acompsmall}
	\frac{\sigma_1-\sigma_2}{d} >  \frac{1}{p_1}- \frac{1}{p_2} .
	\end{equation}
	Then there is some $c>0$ such that 
	\begin{equation}\label{1603-1small}
	  e_k \big(\mbet\hookrightarrow \mbzt\big) %\sim e_k \big(\mbets\hookrightarrow \mbzts\big)
          \le c   k^{-\frac{\sigma_1-\sigma_2}{d}} .  
	\end{equation}	
\end{proposition}

\begin{proof} 

	{\em Step 1.}\quad First we   prove that
	\begin{equation}\label{uu1}
	\lrie (m^{2^{jd}}_{u_1,p_1}\hookrightarrow m^{2^{jd}}_{u_2,p_2})\le c 2^{jd(\frac{1}{r}-\frac{1}{u_1}+\frac{1}{u_2})},\quad \text{whenever}  \quad \frac{1}{r}> \frac{1}{p_1}-\frac{1}{p_2}.  
	\end{equation} 
	Since $p_1<p_2$ we have for any $\nu$, $0\le \nu\le j$,
	\begin{align*}
	  2^{(j-\nu) d (\frac{p_2}{u_2}-1)}&\!\!\!\! \sum_{k:Q_{j,k}\subset Q_{\nu,m}}\!\!\! |\lambda_{j,k}|^{p_2}  \\
          \le 	
	&\ 2^{(j-\nu) d (\frac{p_2}{u_2}-\frac{p_1}{u_1})}\big(\!\!\sup_{k:Q_{j,k}\subset Q_{\nu,m}} |\lambda_{j,k}|\big)^{p_2-p_1}   \;
	2^{(j-\nu) d (\frac{p_1}{u_1}- 1)}\!\!\!\!\sum_{k:Q_{j,k}\subset Q_{\nu,m}} \!\!\!|\lambda_{j,k}|^{p_1}.   
	\end{align*}
	Moreover $\frac{p_2}{u_2}>\frac{p_1}{u_1}$,  so taking the supremum we get
	\begin{equation}\label{uu1a}
	\|\lambda|m^{2^{jd}}_{u_2,p_2}\| \le 2^{jd(\frac{1}{u_2}-\frac{p_1}{p_2u_1})} \|\lambda|\ell^{2^{jd}}_\infty\|^{1-\frac{p_1}{p_2}}  \|\lambda|m^{2^{jd}}_{u_1,p_1}\|^{\frac{p_1}{p_2}} .
	\end{equation}	
	
	Let $m^{*}_{u_2,p_2}$ denote the space $m^{2^{jd}}_{u_2,p_2}$ equipped with the norm 
	\begin{equation}\label{uu1c}\|\lambda|m_{u_2,p_2}^\ast\| =  2^{-jd(\frac{1}{u_2}-\frac{p_1}{p_2u_1})} \|\lambda|m^{2^{jd}}_{u_2,p_2}\| .\end{equation}  
	Thus \eqref{uu1a} implies
	\begin{equation}\label{uu1b}
	\|\lambda|m^\ast_{u_2,p_2}\| \le \left\|\lambda|\ell^{2^{jd}}_\infty\right\|^{1-\frac{p_1}{p_2}}  \left\|\lambda|m^{2^{jd}}_{u_1,p_1}\right\|^{\frac{p_1}{p_2}} .
	\end{equation}	
        
	Now we estimate the entropy numbers  of the embedding $m^{2^{jd}}_{u_1,p_1}\hookrightarrow m^{*}_{u_2,p_2}$. 
	%
	%$m^{2^{jd}}_{u,p_1}\rightarrow \ell^\infty^{2^{jd}}$. 
	Lemma~\ref{lemma15030} implies 
	\begin{align}
	e_k(m^{2^{jd}}_{u_1,p_1}\hookrightarrow \ell_\infty^{2^{jd}}) \le 2^{jd(\frac{1}{p_1}-\frac{1}{u_1})} e_k(\ell^{2^{jd}}_{p_1}\hookrightarrow \ell_\infty^{2^{jd}})
	\end{align}
	So using the inequality \eqref{uu1b},   Sch\"utt's estimates \eqref{Schuett}, Lemma~\ref{dimN} and the interpolation properties of the entropy numbers  \eqref{5.25}, \eqref{ek_ipol} with $\theta=\frac{p_1}{p_2}$, we get
	\begin{align}\nonumber
	  e_{2k-1}&(m^{2^{jd}}_{u_1,p_1}  \hookrightarrow m^{*}_{u_2,p_2})\\
           \le &\ c e_k(m^{2^{jd}}_{u_1,p_1}\hookrightarrow \ell_\infty^{2^{jd}})^{1-\frac{p_1}{p_2}} e_k(m^{2^{jd}}_{u_1,p_1}\hookrightarrow m^{2^{jd}}_{u_1,p_1})^{\frac{p_1}{p_2}}\nonumber  \\
	 \le &\ c 2^{jd(\frac{1}{p_1}-\frac{1}{u_1})(1- \frac{p_1}{p_2})}
	\begin{cases}
	1 &  \text{if}\  1\le k \le \log 2^{jd+1},\\
	\Big(k^{-1}\log \big((1+\frac{2^{jd}}{k})\big)\Big)^{\frac{1}{p_1}-\frac{1}{p_2}} & \text{if}\ \log 2^{jd+1}\le k \le 2^{jd+1},\\
	2^{-k 2^{-jd}} 2^{-jd(\frac{1}{p_1}-\frac{1}{p_2})} & \text{if} \ k \ge 2^{jd+1}.
	\end{cases}
\label{uu2}	\end{align}
	Let us take $r>0$ such that $\frac{1}{r}> \frac{1}{p_1}- \frac{1}{p_2}$. Then  using the estimates \eqref{uu2} we can prove that 
	\begin{equation}\label{uu3}
	k^\frac{1}{r} e_{k}(m^{2^{jd}}_{u_1,p_1}  \hookrightarrow m^{*}_{u_2,p_2}) \le c 
	2^{jd(\frac{1}{r} -\frac{1}{p_1}+\frac{1}{p_2})} 2^{jd(\frac{1}{p_1}-\frac{1}{u_1})(1- \frac{p_1}{p_2})} = 2^{jd(\frac{1}{r} -\frac{p_1}{u_1}(\frac{1}{p_1}-\frac{1}{p_2}))}.
	\end{equation}
	But the relation \eqref{uu1c} between the norms leads to 
	\[
	e_{k}(m^{2^{jd}}_{u_1,p_1}  \hookrightarrow m^{2^{jd}}_{u_2,p_2}) =  2^{jd(\frac{1}{u_2}-\frac{p_1}{p_2u_1})}  e_{k}(m^{2^{jd}}_{u_1,p_1}  \hookrightarrow m^{*}_{u_2,p_2})
	\]	
	So \eqref{uu3} gives 	
	\begin{equation}\label{uu1d}
          k^\frac{1}{r} e_{k}(m^{2^{jd}}_{u_1,p_1}  \hookrightarrow m^{2^{jd}}_{u_2,p_2}) \le c\ 2^{jd(\frac1r-\frac{1}{u_1}+\frac{1}{u_2})}.
        \end{equation} 
	This proves \eqref{uu1}. 	\\
	
	{\em Step 2.}\quad We proceed similar to the proof of Proposition~\ref{slarge}. 
	%%%%%%%%%%%%%%%%%%%%%%%%%%%%%%%%%%%%%%%%%%%%%%%%%%%%%%%%%%%%%%%%%%%%%%%%%%%
	%%%%%%%%%%%%%%%%%%%%%%%%%%%%%%%%%%%%%%%%%%%%%%%%%%%%%%%%%%%%%%%%%%%%%%%%%%%%%%%
	Let $M\in \N$.  We decompose the operator $\id: \mbet\rightarrow \mbzt$ as in \eqref{P_M}.
        %into the sum of two operators 
	%\[ P_M= \sum_{j=0}^M \id_j ,\qquad\text{and}\qquad Q_M= \sum_{j=M+1}^\infty \id_j ,  \]  
	%where 
	%\[ (\id_j \lambda)_{\nu,\ell} := 
	%\begin{cases} 
	%\lambda_{\nu,\ell} & \text{if}\qquad \nu=j, \\
	%0 & \text{if}\qquad \nu\not= j . 
	%\end{cases}\]

	%We take $r_1=p_1$ and $r_2=u_2$ 
	The elementary properties of entropy numbers imply 
	\begin{align}\label{uu1e}
	e_k(\id_j:\mbet\rightarrow \mbzt)  \le  & \ 2^{-j\delta} e_k \big(\mmbet\hookrightarrow \mmbzt\big) \end{align} 
	where $\delta=\sigma_1-\sigma_2-\frac{d}{u_1}+\frac{d}{u_2}$. Consequently, using \eqref{uu1}, \eqref{uu1d} and \eqref{uu1e} we arrive at 
	\begin{equation}\label{uu4}
	\lrie(\id_j) \le C   2^{-j\delta}  2^{jd(\frac{1}{r}-\frac{1}{u_1}+\frac{1}{u_2})}  = C 2^{jd(\frac{1}{r} - \frac{\sigma_1-\sigma_2}{d})},   \qquad \frac{1}{r} >  \frac{1}{p_1}-\frac{1}{p_2}.   
	\end{equation}
	%cf. \eqref{3-2-16}. 
	%We recall, that for any $r>0$ there exists a positive number $\rho\le 1$ such that 
	%\begin{equation}\label{3-2-19l}
	%\lrie (\sum_j T_j)^\varrho \le \sum_{j}  \lrie (T_{j})^\varrho\, , 
	%\end{equation} 
	%cf. K\"onig, \cite[1.c.5]{Koe}.
	Hence for some  positive number  $\rho\le 1$,  \eqref{uu4} and \eqref{3-2-19l} yield
	\begin{align*} %\label{3-2-20l}
	\lrie (P_M)^\varrho  \le  \sum_{j=0}^M \lrie (\id_{j})^\varrho 
	%\nonumber
	%\\
	\le  c\, 
	\sum_{j=0}^M  2^{\varrho jd(\frac{1}{r} - \frac{\sigma_1-\sigma_2}{d})} \le c' 2^{\varrho Md(\frac{1}{r} - \frac{\sigma_1-\sigma_2}{d})} ,
	\end{align*}
	if we choose $r$ such that  
%	\[
$\	\frac{1}{r} > \frac{\sigma_1-\sigma_2}{d} > \frac{1}{p_1}-\frac{1}{p_2}$. 
%	\]
	Thus 
	\begin{equation}\label{1603-3la}
	e_{2^{Md}} (P_M) \le C 2^{-Md(\frac{\sigma_1-\sigma_2}{d})} .
	\end{equation}
	
	Similarly, 
	\begin{align*} %\label{3-2-21l}
	\lrie (Q_M)^\varrho  \le  \sum_{j=M+1}^\infty \lrie (\id_{j})^\varrho 
	%\nonumber
	%\\
	\le  c\, 
	\sum_{j=M+1}^\infty  2^{\varrho jd(\frac{1}{r} - \frac{\sigma_1-\sigma_2}{d})} \le c' 2^{\varrho Md(\frac{1}{r} - \frac{\sigma_1-\sigma_2}{d})} ,
	\end{align*}
	if we choose $r$ this time such that  
%	\[
$\ 	0 \le  \frac{1}{p_1}-\frac{1}{p_2} <\frac{1}{r}< \frac{\sigma_1-\sigma_2}{d}$. 
%	\]
	Thus 
	\begin{equation}\label{1603-4la}
	e_{2^{Md}} (Q_M) \le C 2^{-Md(\frac{\sigma_1-\sigma_2}{d})} .
	\end{equation} 
	Now \eqref{1603-3la} and \eqref{1603-4la} give us 
%	\begin{equation*}%\label{1603-5}
$\ 	e_{2^{Md}} (\id: \mbet\rightarrow \mbzt) \le C 2^{-Md(\frac{\sigma_1-\sigma_2}{d})} $
%	\end{equation*} 
	and by standard arguments
%	\begin{equation*}%\label{1603-5l}
$	e_{k} (\id: \mbet\rightarrow \mbzt) \le C k^{-\frac{\sigma_1-\sigma_2}{d}}$, $ k\in\nat$,
%	\end{equation*} 
as desired.
\end{proof}

%%%%%%%%%%%%%%%%
Now we can present the (almost) complete counterpart of Theorem~\ref{ek_classical} showing some surprising phenomenon.

\begin{theorem}\label{reallymorreycase}
Let  $\sigma_i\in \R$, $0<q_i\leq\infty$, $0<p_i\leq u_i<\infty$, $i=1,2$, and   
{\begin{equation}\label{bd3acompn}
\frac{\sigma_1-\sigma_2}{d} >  \frac{p_1}{u_1} \Big(\frac{1}{p_1}- \frac{1}{p_2}\Big)   >  \max\bigg\{0,\frac{1}{u_1} - \frac{1}{u_2}\bigg\}.
\end{equation}}
\bli
\item[{\upshape\bfseries (i)}] If $\ \frac{\sigma_1- \sigma_2}{d} >  \frac{1}{p_1} -\frac{1 }{p_2}$, then
 \begin{equation}\label{morreyslarge}
    e_k \big(\mbet\hookrightarrow \mbzt\big) \sim     k^{-\frac{\sigma_1-\sigma_2}{d}} .  
\end{equation}

\item[{\upshape\bfseries (ii)}]
  If $\ \frac{p_1}{u_1} \Big(\frac{1}{p_1}- \frac{1}{p_2}\Big)< \frac{\sigma_1- \sigma_2}{d} \le  \frac{1}{p_1} -\frac{1 }{p_2}$, then there exists some $c>0$ and for any $\varepsilon>0$ some $c_\varepsilon>0$ such that for all $k\in\nat$, 
\begin{align} \label{morreyssmall-2}
  c k^{- \frac{u_1}{u_1-p_1}(\frac{\sigma_1-\sigma_2}{d}-\frac{p_1}{u_1}(\frac{1}{p_1}- \frac{1}{p_2}) )} & \le    e_k \big(\mbet\hookrightarrow \mbzt\big) \\
  &\le c_\varepsilon k^{- \frac{u_1}{u_1-p_1}(\frac{\sigma_1-\sigma_2}{d}-\frac{p_1}{u_1}(\frac{1}{p_1}- \frac{1}{p_2}) )+\varepsilon}.
  %k^{-\alpha (\frac{s_1-s_2}{d}-\frac{p_1}{u_1}(\frac{1}{p_1}- \frac{1}{p_2}) )+\varepsilon} ,  
\nonumber\end{align}
%where $\alpha$ is given by \eqref{alpha}.
\eli
\end{theorem}

\begin{proof}
Recall that the condition \eqref{bd3acompn}, that is,  
\begin{equation}
\max\Big\{0,\frac{1}{u_1} - \frac{1}{u_2}\Big\} % <  \frac{1}{p_1}-\frac{1}{u_2} 
< \frac{p_1}{u_1} \Big(\frac{1}{p_1}- \frac{1}{p_2}\Big)  %< \frac{s_1-s_2}{d}
\end{equation}
is equivalent to  $p_1 < p_2$   and $\frac{p_2}{u_2} > \frac{p_1}{u_1}$. All other cases are already covered by Theorem~\ref{ek_classical}.

{\em Step 1.}\quad We prove (i).
The statement follows immediately from Proposition~\ref{slargebis} and Lemma~\ref{lemma1503}. %\\

{\em Step 2.}\quad We first show the upper estimate in (ii). Here we make use of the interpolation result, Theorem~\ref{ThmInterpolation}, in its sequence space version, together with the interpolation property of entropy numbers \eqref{ek_ipol}.

Let $\varepsilon_1>0$ and $\tau_1=\sigma_2+d(\frac{1}{p_1}-\frac{1}{p_2})+d\varepsilon_1$ be such that $\tau_1 -\sigma_2 > d(\frac{1}{p_1}-\frac{1}{p_2})$. Hence  we can apply (i) to the compact embedding 
%\[
$\id_1: \widetilde{n}^{\tau_1}_{u_1,p_1,v_1} \hookrightarrow \mbzt$, %\quad \tau_1=t_1+\frac{d}{2},
%\] 
where $0<v_1\leq\infty$ is arbitrary. Thus \eqref{morreyslarge} leads to
\beq
e_k(\id_1: \widetilde{n}^{\tau_1}_{u_1,p_1,v_1} \hookrightarrow \mbzt)\leq \ c\ k^{-\frac{\tau_1-\sigma_2}{d}} = c_{\varepsilon_1} \ k^{-(\frac{1}{p_1}-\frac{1}{p_2})- \varepsilon_1}.
\eeq
We choose $\ \tau_2=\sigma_2+d \frac{p_1}{u_1}(\frac{1}{p_1}-\frac{1}{p_2})\ $ and $v_2$ with %\[
$\  0<v_2\leq \min(1,\max(1,\frac{u_2}{u_1})\frac{p_1}{p_2})q_2$. % \]
Then by \cite[Thm.~3.1]{hs12b} the embedding
$\id_2 : \widetilde{n}^{\tau_2}_{u_1,p_1,v_2} \hookrightarrow \mbzt $ %,\quad \tau_2=t_2+\frac{d}{2}, 
is continuous, that is,

$\ e_k(\id_2: \widetilde{n}^{\tau_2}_{u_1,p_1,v_2} \hookrightarrow \mbzt)\leq \ c$.
Now we determine the number $\theta\in (0,1)$ appropriately such that
$\sigma_1 = (1-\theta)\tau_1 + \theta \tau_2$, 
%\eeq
which is always possible in case (ii). The sequence space version of Theorem~\ref{ThmInterpolation}(i) yields then
\beq
\mbet = \left(\widetilde{n}^{\tau_1}_{u_1,p_1,v_1}, \ 
\widetilde{n}^{\tau_2}_{u_1,p_1,v_2}\right)_{\theta, q_1}\ .
\eeq
Thus the interpolation property \eqref{ek_ipol} leads to
\[
e_{2k}(\mbet\hookrightarrow \mbzt) \leq \ c \ e_k(\id_1)^{1-\theta} e_k(\id_2)^\theta \leq \ c'_{\varepsilon_1}\ k^{-(1-\theta)(\frac{1}{p_1}-\frac{1}{p_2} + \varepsilon_1)}
\]
and it remains to show that %$\alpha$ given by \eqref{alpha} satisfies
\[
-\frac{u_1}{u_1-p_1} \left(\frac{\sigma_1-\sigma_2}{d}-\frac{p_1}{u_1}\left(\frac{1}{p_1}- \frac{1}{p_2} \right) \right)+\varepsilon = -(1-\theta)\left(\frac{1}{p_1}-\frac{1}{p_2} +\varepsilon_1\right)
\]
and appropriately chosen $\varepsilon$ (in dependence on $\varepsilon_1$). {Moreover, $\varepsilon_1\rightarrow 0$ implies $\varepsilon\rightarrow 0$}.  However, 
this is obvious by straightforward calculation.

%%%%%%%%%%%%%%%%%%%%%%%%%%%%%%%%%%%%%%%%%%%%%%%%%%%%%%%%%%%%%%%%%%% 
%%%%%%%%%%%%%%%%%%%%%%%%%%%%%%%%%%%%%%%%%%%%%%%%%%%%%%%%%%%%%%%%%%% 
{\em Step 3.}\quad It remains to show the lower estimate in (ii).
%, the one in (ii) is a consequence of \eqref{ek_below_class}. 
Let $k_j=\whole{2^{jd(1-\frac{p_1}{u_1})}}$. We consider the following commutative diagram
\[ 
\begin{CD}
\ell^{k_j}_{\infty} @>\id>> \ell^{k_j}_{p_2}\\
@V{P_1}VV @AA{P_2}A\\
\mmbet @>\id>> \mmbzt\, .
\end{CD}
 \]

The operator $P_1$ is defined in the following way. Let $\lambda^{(j)}=\big(\lambda^{(j)}_\ell\big)$ be a sequence constructed in Substep 2.2 of the proof of Theorem 2.1 in \cite{HaSk-sm}. The sequence has $k_j$ elements equal to $1$ and the remaining elements are $0$. Moreover $\|\lambda^{(j)}|\mmbet\|=1$.  Let $T$ be a one-to-one correspondence between the set $\{1,\ldots , k_j\}$ and the cubes $Q_{0,\ell}$ such that $\lambda^{(j)}_\ell=1$. For  $\mu = (\mu_m)\in \ell^{k_j}_{\infty}$ we put
\[
P_1(\mu)_\ell = 
\begin{cases} 
\mu_m & \text{if}\qquad T(m)=Q_{0,\ell}\\
0 & \text{otherwise}. 
\end{cases}
\]
Then $P_1$ is a linear operator mapping $\ell^{k_j}_{\infty}$ into $\mmbet$ and it norm equals $1$.  

The operator $P_2$ is a projection,  $P_2(\lambda)_m = \lambda_\ell$ if $T(m)=Q_{0,\ell}$. Then Lemma~\ref{lemma15030} implies that $\|P_2\|\le 2^{jd(\frac{1}{p_2}-\frac{1}{u_2})}$.  We conclude from the above diagram and Sch\"utt's estimates that 
\begin{equation}
c\ 2^{jd(\frac{1}{u_2}- \frac{p_1}{u_1p_2})} \le e_{k_j}(\mmbet\hookrightarrow \mmbzt). 
\end{equation} 
This finally leads to 
 \begin{align*}
   c\ k_j^{- \frac{u_1}{u_1-p_1}(\frac{\sigma_1-\sigma_2}{d}-\frac{p_1}{u_1}(\frac{1}{p_1}- \frac{1}{p_2}) )} & \ = c\ 2^{-jd(\frac{\sigma_1-\sigma_2}{d}-\frac{p_1}{u_1}(\frac{1}{p_1}- \frac{1}{p_2}) )} \\
   & \ \le    e_{k_j} \big(\mbet\hookrightarrow \mbzt\big) .
 \end{align*}
 \end{proof}

\begin{ownremark}
  Note that the `classical' exponent $\ \frac{\sigma_1-\sigma_2}{d}\ $ of the asymptotic decay of entropy numbers, as it appears in Theorem~\ref{ek_classical} already, is extended in part (i) of the above theorem for $\sigma_1-\sigma_2$ `large enough', that is, $\ \sigma_1-\sigma_2 > d(\frac{1}{p_1}-\frac{1}{p_2})$. However, for small differences $\sigma_1-\sigma_2$ and in the proper Morrey case for the source space as dealt with in (ii),  
which does not exist when $p_1=u_1$, we can at the first time observe a completely different behaviour for the asymptotic decay of the entropy numbers. In particular, if 
%\[
$\frac{p_1}{u_1} \big(\frac{1}{p_1}- \frac{1}{p_2}\big) <\frac{\sigma_1-\sigma_2}{d}<  \frac{1}{p_1} -\frac{1}{p_2}, $
%\]
then 
% \[
$\frac{u_1}{u_1-p_1}\left(\frac{\sigma_1-\sigma_2}{d}-\frac{p_1}{u_1}\left(\frac{1}{p_1}- \frac{1}{p_2}\right) \right) < \frac{\sigma_1-\sigma_2}{d}$. %\]
So in that case the power $\frac{\sigma_1-\sigma_2}{d}$ does not describe the asymptotic behaviour of entropy numbers, and their decay is much slower than usual. %\\

On the other hand, if $\ \sigma_1-\sigma_2=d(\frac{1}{p_1}-\frac{1}{p_2})$, that is, when the above cases (i) and (ii) meet, then the upper estimate in \eqref{morreyssmall-2} reads as
\[
  e_k \big(\mbet\hookrightarrow \mbzt\big) \le c_\varepsilon   k^{-\frac{\sigma_1-\sigma_2}{d}+\varepsilon} ,  
\]
that is, the same exponent (up to $\varepsilon$) as in \eqref{morreyslarge}. This gives some hint that the  exponent in \eqref{morreyssmall-2} might be close to the precise description. We postpone some further discussion to Remark~\ref{rem-limiting} below.
%Also \eqref{morreys_p2u2} supports this assumption.
 %\open{more/less discussion about similar phenomena, say in weighted cases or of generalised smoothness ?}
\end{ownremark} 

	%%%%%%%%%%%%%%%%%%%%%%%%%%%%%%%%%%%%%%%%%%%%%%%%%%%%%%%%%%%%%%%%%%%%%

%%%%%%%%%%%%%%%%%%%%%%%%%%%%%%%%%%%%%%%%%

\section{Entropy numbers of compact embeddings in smoothness Morrey spaces on bounded domains}\label{entro_func}
%%%%%%%%%%%%%%%%%%%%%%%%%%%%%%%%%%%%%%%%%%%%%%
{First we note that the asymptotic behaviour of the entropy numbers $e_k \big(\id_{\mathcal{A}}: \MAe(\Omega)\hookrightarrow \MAz(\Omega)\big)$ for an  arbitrary domain, i.e., an  open bounded set $\Omega$ in $\Rn$, is the same as for a $C^\infty$ bounded domain. So we can concentrate on  $C^\infty$ bounded domains afterwards.%in further argumentation. 

\begin{lemma}\label{entlemma}
Let $\Omega$ be a non-empty open bounded set and $B_e(0,1)$ be an open unit ball in $\R^d$. 
%Let $A_1(\R^d)$ and $A_2(\R^d)$ be  quasi-Banach spaces such that $\mathcal{S}(\R^d)\hookrightarrow A_1(\R^d)\hookrightarrow %A_2(\R^d)\hookrightarrow \mathcal{S}'(\R^d)$. We assume  that the translations, dilations and  multiplication by the smooth %compactly supported  functions defined bounded operators mapping $A_i(\R^d)$ into itself, $i=1,2$. Let the spaces $A_i(\Omega)$, $i=1,2$ be defined by restrictions and equipped with the usual norm.     
Then the embedding $\MAe(\Omega)\hookrightarrow \MAz(\Omega)$ is compact if and only if the embedding $\MAe(B_e(0,1))\hookrightarrow \MAz(B_e(0,1))$ is compact. Moreover, in that case, 
\begin{equation}\label{entlemma1}%\nonumber
e_k\big(\MAe(\Omega)\hookrightarrow \MAz(\Omega)\big)\sim e_k\big(\MAe(B_e(0,1))\hookrightarrow \MAz(B_e(0,1)))\big) .
\end{equation}  
\end{lemma}

\begin{proof}
The lemma holds for any ball  $\Omega=B_e(x,R)$, centred at $x\in \R^d$ with radius $R>0$, since translations and dilations define bounded operators in the spaces $\MA(\R^d)$. To  prove the  statement for general $\Omega$ it is sufficient to prove \eqref{entlemma1} since an embedding is compact if, and only if, its entropy numbers $e_k$ tend to zero for $k\rightarrow \infty$. 

Let $B_e(0,R)$ be an open  ball such that  $\overline{\Omega} \subset B_e(0,R) $. Multiplication by a smooth compactly supported function defines a bounded operator in $\MA(\R^d)$,  therefore the expression  
\[  \|f|\MA(\Omega)\|_* = \inf\left\{ \|g|\MA(\R^d)\|: \quad  f=g|_\Omega\quad \text{and}\quad \supp g\subset B_e(0,R)\right\}\] 
is an equivalent norm in $\MA(\Omega)$. 
Let $B_i(g,r)$  denote a ball in  the space $\mathcal{A}^{s_i}_{u_i,p_i,q_i}(\Omega)$ centred at $g$ with radius $r$, and let $\widetilde{B}_i(g,r)$  denote a corresponding ball in $\mathcal{A}^{s_i}_{u_i,p_i,q_i}(B_e(0,R))$, $i=1,2$.

Let $c>1$ and $f\in B_1(0,1)$. Then there exists some $g\in \MAe(B_e(0,R))$ such that $f=g|_\Omega$ and  $\|g|\MAe(B_e(x,R))\| \le c\|f|\MAe(\Omega)\|_* < c$. Let $\varepsilon>0$ and choose $k\in\nat$  sufficiently large such that
\[e_k\big(\MAe(B_e(x,R))\hookrightarrow \MAz(B_e(x,R))\big)  <\varepsilon. \]
The definition of entropy numbers implies that there are $g_j\in  \MAz(B_e(x,R))$, $j=1,\ldots, 2^{k-1}$, such that
\begin{equation}\label{entlemma2}
\widetilde{B}_1(0,c)\subset \bigcup_{j=1}^{2^{k-1}} \widetilde{B}_2(g_j, c\varepsilon). 
\end{equation}
Thus there exists $g_m$, $m\in \{1, \dots, 2^{k-1}\}$, with $\|g-g_m| \MAz(B_e(x,R))\|< c\varepsilon$. Then 
%\begin{equation*}
$\| f -g_m|_\Omega|\MAz(\Omega) \|_* %= \| (g -g_m)|_\Omega|\MAz(\Omega) \|_* 
\le \| g -g_m|\MAz(B_e(x,R)) \|\le c\varepsilon$.
%\end{equation*} 
 This implies that there exists a positive constant $C$ independent of $k$ such that $e_k\big(\MAe(\Omega)\hookrightarrow \MAz(\Omega)\big)\le C e_k\big(\MAe(B_e(0,1))\hookrightarrow \MAz(B_e(0,1))\big)$.

To prove the opposite inequality we can take a ball $B(x,r)$ such that  $\overline{B(x,r)}\subset \Omega$ and argue in a similar way as above. 
\end{proof}}
%%%%%%%%%%%%%%%%%%%%%%%%%%%%%%%%%%%%%

Now we can {concentrate on} %are interested in 
the entropy numbers of the compact embedding characterised by Theorem~\ref{comp}.
\ignore{\begin{equation}
  \label{id_comp_a}
  \id_{\mathcal{A}} : \MAe(\Omega) \hookrightarrow  \MAz(\Omega), 
 \end{equation} 
where $s_i\in \R$, $0<q_i\leq\infty$, $0<p_i\leq u_i<\infty$, $i=1,2$,  with
\begin{equation}
  \label{cond-comp}
\frac{s_1-s_2}{d} > \max\bigg\{0,\frac{1}{u_1} - \frac{1}{u_2}, 
\frac{p_1}{u_1} \Big(\frac{1}{p_1}- \frac{1}{p_2}\Big) \bigg\},
\end{equation}
recall Theorem~\ref{comp}.  }

{\em Preparation.}\quad Recall that $\Omega\subset\Rn$ is a bounded $C^\infty$ domain and the spaces are defined by restriction, see Definition~\ref{D-spaces-Omega}. Let $\Omega_t= \{x\in\Rn: \dist(x, \Omega)< t\}$, for some $t>0$.  We choose a dyadic cube $Q$ (applying some appropriate dilations or translations first, if necessary) such that 
%$$
$\ \supp \psi_{i,\nu,m} \subset Q\ $  \text{if}  $\ \supp \psi_{i,\nu,m} \cap \Omega_t \not = \emptyset$,   
%$$
and $\ 
%$$
\supp \phi_{0,m} \subset Q\ $ \text{if} $\  \supp \phi_{0,m}\cap \Omega_t \not = \emptyset$.  
%$$
Let $h\in C^\infty_0(\Rn)$ be a test function such that $\supp h \subset \Omega_t$  and $h(x)=1$ for any $x\in \Omega_{t/2}$. We choose $N\in \N$ such that 
$N^{-1}\le  p_i\le u_i $, $N^{-1}< q_i$ and $|s_i|< N $, $i= 1,2$.  Then there exists an extension operator $\ext_N$, common for both spaces 
$\MBe(\Omega)$ and $\MBz(\Omega)$, with  $\re \circ M_h = \re$, that is, 
$\re \circ M_h \circ \ext_N = \id $ on ${\cal N}^{s_i}_{u_i,p_i,q_i}(\Omega)$, cf. \cite{Saw2010}. Here $M_h$ 
is the natural pointwise multiplication mapping $\ M_h: f \ni  \mathcal{S}(\Rn) \mapsto h \cdot f$.
Now using the wavelet characterisation of the spaces $\MB(\R^d)$ by Daubechies wavelets, cf. Theorem~\ref{wavemorrey},   in particular the isomorphism $T$ defined by  \eqref{wavemorreyrem1}, one can  factorise the embeddings for function spaces via the embeddings of the corresponding sequence spaces. More precisely, the following diagram is commutative 
\[ 
\begin{CD}
\MBe(\Omega )@>\id>> \MBz(\Omega )\\
@V{T\circ M_h\circ\ext_N}VV @AA{\re\circ T^{-1}}A\\
 \mbet( Q )@>\id>> \mbzt(Q)\, ,
\end{CD}
 \]
 with $\sigma_i=s_i+\frac{d}{2}$. Thus the multiplicativity of entropy numbers \eqref{e-multi} immediately yields
 \[
 e_k(\id:\MBe(\Omega) \to \MBz(\Omega))\leq \ c\ e_k(\id:\mbet(Q)\to \mbzt(Q)).
 \]
 Conversely, by the diffeomorphic properties of Besov-Morrey spaces, using translations and dilations if necessary we can assume that the domain $\Omega$ satisfies the following conditions: there exists $\nu_0\ge 0$ such that $Q_{\nu_0,0}\subset \Omega$, \quad and $\ \supp \psi_{i,\nu, m} \subset \Omega$ if $\ Q_{\nu,m} \subset Q_{\nu_0,0}$. 
%, $\nu\ge 0$, {then} ,\quad  and if $\ Q_{0,m} \subset Q_{\nu_0,0}$, {then} $\ \supp \phi_{0, m} \subset \Omega$,  cf. \cite{Saw2010}. 
Let $\widetilde{T}$ denote the restriction of the isomorphism $T^{-1}$ to the subspaces $\mb(Q_{\nu_0,0}) \subset \mb(\Rn)$,  cf. \eqref{wavemorreyrem1}. 
Then $\ \widetilde{T}(\lambda)\in  \MB(\Rn)$ and  $\supp \widetilde{T}(\lambda)\subset \Omega$  \text{for any} $\ \lambda\in \mb(Q_{\nu_0,0})$.  Moreover, there exists a positive $\varepsilon>0$ such that for any $\lambda\in \mb(Q_{\nu_0,0})$, $\dist(\supp \widetilde{T}(\lambda),\Rn\setminus \Omega) > \varepsilon$.  If we denote by $P: \mb(\Rn)\rightarrow \mb(Q_{\nu_0,0})$  the usual projection in sequence spaces, then we get the  following commutative diagram  
\[ 
\begin{CD}
\mbet( Q_{\nu_0,0} )@>\id>> \mbzt(Q_{\nu_0,0})\\
@V{\re \circ \widetilde{T}}VV @AA{P\circ T\circ\ext_N}A\\
 \MBe(\Omega )@>\id>> \MBz(\Omega )\, .
\end{CD}
 \]
Thus obvious arguments immediately yield
 \begin{align*}
 e_k(\id:\mbet(Q)\to \mbzt(Q)) &\sim e_k(\id:\mbet(Q_{\nu_0,0})\to \mbzt(Q_{\nu_0,0}))  \\
 & \le \ c e_k(\id:\MBe(\Omega) \to \MBz(\Omega))\ .
 \end{align*}

Please note that we may take the same system of wavelets, and thus the same operator $\widetilde{T}$, for $\MBe(\Rn)$ and $\MBz(\Rn)$. 

Hence upper and lower estimates for the entropy numbers of sequence space embeddings lead immediately to upper and lower estimates for the  corresponding function space embeddings, using the multiplicativity of entropy numbers \eqref{e-multi} again. So in view of our results in Section~\ref{entro_seq} and the above arguments we have the following results.

%%%%%%%%%%%%%%%%%%%%%%%%%%%%%%%%%

%\begin{corollary}\label{ek_classical_fu}
\begin{theorem}\label{ek_classical_fu}
{Let $\Omega$ be an arbitrary bounded domain in $\R^d$. }	Let  $s_i\in \R$, $0<q_i\leq\infty$, $0<p_i\leq u_i<\infty$, $i=1,2$. Assume that 
\bli
\item[{\hfill\bf (i)\hfill}] either
	\begin{equation*}%\label{bd3acomp-0}
	\frac{s_1-s_2}{d} > \max\bigg\{0,\frac{1}{u_1} - \frac{1}{u_2}, 
	\frac{p_1}{u_1} \Big(\frac{1}{p_1}- \frac{1}{p_2}\Big) \bigg\}= 0,
	\end{equation*}
\item[{\hfill\bf (ii)\hfill}] or
\begin{equation*}%\label{bd3acompu}
\frac{s_1-s_2}{d} > \max\bigg\{0,\frac{1}{u_1} - \frac{1}{u_2}, 
\frac{p_1}{u_1} \Big(\frac{1}{p_1}- \frac{1}{p_2}\Big) \bigg\} = \frac{1}{u_1} - \frac{1}{u_2}>0.
\end{equation*}
\eli
Then
\begin{equation}\label{ek-class-fu}%\label{1603-1n}%\label{1603-1-0}
    e_k \big(\id_{\mathcal{A}}: \MAe(\Omega)\hookrightarrow \MAz(\Omega)\big) \sim  k^{-\frac{s_1-s_2}{d}} .  
\end{equation}
\end{theorem}

\begin{proof}
{In view of Lemma \ref{entlemma} it is sufficient to consider a $C^\infty$ domain.} In the Besov-Morrey situation, that is, when $\mathcal{A}=\mathcal{N}$, Theorem~\ref{ek_classical} together with our preceding remarks covers the results. Otherwise, when $\mathcal{A}=\mathcal{E}$, then \eqref{elem} (which remains true in the situation when $\Rn$ is replaced by $\Omega$) and the independence of all above assumptions of the fine parameters $q_i$, $i=1,2$, implies the assertions, where we benefit again from the multiplicativity property of entropy numbers \eqref{e-multi}.
\end{proof}

\begin{ownremark}\label{remark-ek-class}
If $u=p$, then $ \MA(\Omega)=\A(\Omega)$, so we can compare our above result with the well-known situation  of the embedding
$%\[
\id_A: 
\Ae(\Omega) \to \Az(\Omega)$, 
%\]
where $s_i\in\R$, $ 0<p_i, q_i\leq\infty$ 
($p_i<\infty $ in the $F$-case), $i=1,2$, and the spaces $\ \A(\Omega)\ $ are defined by restriction. Then $\id_A$ is compact, if, and only if, 
%\[  
$\frac{s_1-s_2}{d} > \max\big\{0,\frac{1}{p_1} - \frac{1}{p_2} \big\}$,
%\]
recall  also \eqref{cond-comp}. Note that in this situation the above cases (i) and (ii) of Theorem~\ref{ek_classical_fu} cover all possible cases of compactness already. {Edmunds} and {Triebel} proved in \cite{ET1,ET2} (see also \cite[Thm. 3.3.3/2]{ET}) that  
\[
e_k(\id_A: \Ae(\Omega) \hookrightarrow \Az(\Omega)) \ \sim \ k^{-\frac{s_1-s_2}{d}},\quad k\in\N,
\]
which perfectly coincides with our findings \eqref{ek-class-fu}.
\end{ownremark}

We return to the non-classical case, that is, when $p_1<u_1$ or $p_2<u_2$.

\begin{theorem}\label{reallymorrey_fu}
{Let $\Omega$ be an arbitrary bounded domain in $\R^d$. } Let  $s_i\in \R$, $0<q_i\leq\infty$, $0<p_i\leq u_i<\infty$, $i=1,2$, and   
{\begin{equation}\label{bd3acompfu}
\frac{s_1-s_2}{d} >  \frac{p_1}{u_1} \Big(\frac{1}{p_1}- \frac{1}{p_2}\Big)   >  \max\bigg\{0,\frac{1}{u_1} - \frac{1}{u_2}\bigg\}.
\end{equation}}
%Then the embedding 
%\begin{equation} \label{bd1compn}
%	 \mbet\hookrightarrow \mbzt
%\end{equation}
%is compact.
\bli
\item[{\upshape\bfseries (i)}] If $\ \frac{s_1- s_2}{d} >  \frac{1}{p_1} -\frac{1 }{p_2}$, then
\begin{equation}%\label{morreyslarge}
    e_k \big(\id_{\mathcal{A}}: \MAe(\Omega)\hookrightarrow \MAz(\Omega)\big) \sim     k^{-\frac{s_1-s_2}{d}} .  
\end{equation}
\item[{\upshape\bfseries (ii)}]
  If $\ \frac{p_1}{u_1} \Big(\frac{1}{p_1}- \frac{1}{p_2}\Big)< \frac{s_1- s_2}{d} \le  \frac{1}{p_1} -\frac{1 }{p_2}$, then there exists some $c>0$ and for any $\varepsilon>0$ some $c_\varepsilon>0$ such that for all $k\in\nat$, 
\begin{equation}\label{new_morrey_fu}
  c k^{- \frac{u_1}{u_1-p_1}(\frac{s_1-s_2}{d}-\frac{p_1}{u_1}(\frac{1}{p_1}- \frac{1}{p_2}) )} \le    e_k \big(\id_{\mathcal{A}}\big) \le c_\varepsilon k^{- \frac{u_1}{u_1-p_1}(\frac{s_1-s_2}{d}-\frac{p_1}{u_1}(\frac{1}{p_1}- \frac{1}{p_2}) )+\varepsilon}.
  %k^{-\alpha (\frac{s_1-s_2}{d}-\frac{p_1}{u_1}(\frac{1}{p_1}- \frac{1}{p_2}) )+\varepsilon} ,  
\end{equation}
%where $\alpha$ is given by \eqref{alpha}.
\eli
\end{theorem}

\begin{proof}
We follow the same strategy as in case of Theorem~\ref{ek_classical_fu}, using this time Theorem~\ref{reallymorreycase}.
\end{proof}

\begin{ownremark}\label{rem-limiting}
We think that the outcome is really remarkable for at least two reasons: although the result in (ii) is not yet sharp, it shows that the `classical' asymptotic behaviour of the entropy numbers, $e_k\sim k^{-\frac{s_1-s_2}{d}}$, cannot be true in this case. This is indeed surprising and was not to be expected before. Secondly, we want to point out again the interplay between the Morrey fine parameters $p_i$ and the smoothness parameters $s_i$ in this case.
  
  We briefly return to our remark at the end of Section~\ref{entro_seq} concerning the situation $s_1-s_2=d(\frac{1}{p_1}-\frac{1}{p_2})$ where the above cases (i) and (ii) of Theorem~\ref{reallymorrey_fu} meet and we obtain the two-sided estimate
  \[
c\  k^{-\frac{s_1-s_2}{d}} \leq  e_k \big(\id_{\mathcal{A}}: \MAe(\Omega)\hookrightarrow \MAz(\Omega)\big)\leq c_\varepsilon\ k^{-\frac{s_1-s_2}{d}+\varepsilon}.
  \]
  In similar {\em limiting cases} for embeddings an additional $\log$-term appears occasionally, and even an influence of the parameters $q_i$, $i=1,2$, and a different behaviour for the case $\mathcal{A}=\mathcal{N}$ or $\mathcal{A}=\mathcal{E}$ has been found, see \cite{KLSS1,Leo-poznan,HS1}. But we have no claim for our   limiting situation at the moment. Moreover, we have observed a different phenomenon in \cite{HS2} concerning some weighted situation: in that case the local behaviour of the weight was involved in the characterisation of the compactness of that embedding, while the entropy numbers `ignored' this local part completely. So it remains a tricky business after all with many open questions.
\end{ownremark}

\begin{ownremark}
  Note that  Theorems~\ref{ek_classical_fu} and \ref{reallymorrey_fu} imply, in particular, the extension of 
Theorem~\ref{comp} to arbitrary bounded domains in view of the properties of entropy numbers. 
\end{ownremark}

We consider some special cases, and begin with the case $p_1=u_1$, that is, when the source space $\MAe(\Omega)$ coincides with the space $\Ae(\Omega)$. Then Theorem~\ref{reallymorrey_fu}, i.e. \eqref{bd3acompfu} is impossible and we arrive in all possible settings at the classical behaviour for the entropy numbers, recall Remark~\ref{remark-ek-class}.

\begin{corollary}\label{source_class}
  Let $\Omega$ be a bounded domain in $\R^d$,  $s_i\in \R$, $0<q_i\leq\infty$, $0<p_i\leq u_i<\infty$, $i=1,2$, and \eqref{cond-comp} satisfied.  Assume $p_1=u_1$. Then
  \[
e_k \big(\id: \Ae(\Omega)\hookrightarrow \MAz(\Omega)\big) \sim  e_k \big(\id_{A}: \Ae(\Omega)\hookrightarrow \Az(\Omega)\big) \sim     k^{-\frac{s_1-s_2}{d}}
  \]
%\begin{equation}%\label{morreyslarge}
%    e_k \big(\id: \Ae(\Omega)\hookrightarrow \MAz(\Omega)\big) \sim     k^{-\frac{s_1-s_2}{d}} .  
%\end{equation}
\end{corollary}

\begin{proof}
This follows immediately from Theorem~\ref{ek_classical_fu} when $p_1=u_1$.
\end{proof}

\begin{ownremark}
  This `negligence' of the (larger) Morrey target space $\MAz(\Omega)$ (compared with $\Az(\Omega)$) is not completely new: it has its counterpart in the related compactness result \cite[Cor.~4.2]{hs12b}.
  % So in view of Remark~\ref{remark-ek-class} we observe that
  %\[
%e_k \big(\id: \Ae(\Omega)\hookrightarrow \MAz(\Omega)\big) \sim  e_k \big(\id_{A}: \Ae(\Omega)\hookrightarrow \Az(\Omega)\big) %\sim     k^{-\frac{s_1-s_2}{d}}
 % \]
  %whenever $p_2\leq u_2$ and \eqref{cond-comp} holds.  
  Thus the really new and surprising asymptotic behaviour \eqref{new_morrey_fu} of $e_k(\id_{\mathcal{A}})$ can only appear when $p_1<u_1$. 
\end{ownremark}

We finally consider the case $p_2=u_2$ and may also cover the $L_r$ scale, including the case $r=\infty$ described in Proposition~\ref{MorreyintoLinfty}. 
%, completing the result in Corollary~\ref{ek-MorreyintoLinfty} when $r=\infty$
In that situation case (ii) of Theorem~\ref{ek_classical_fu} is excluded and \eqref{cond-comp} reads as
 \begin{equation}\label{cond-comp-A}
\frac{s_1-s_2}{d} > \max\bigg\{0, \frac{p_1}{u_1} \Big(\frac{1}{p_1}- \frac{1}{p_2}\Big) \bigg\} = \frac{p_1}{u_1}\left(\frac{1}{p_1}-\frac{1}{p_2}\right)_+.
\end{equation}

\begin{corollary}\label{ek_into_classical_fu}
Let $\Omega$ be a bounded domain in $\R^d$,  $s_i\in \R$, $0<q_i\leq\infty$, $0<p_i\leq u_i<\infty$, $i=1,2$, and \eqref{cond-comp-A}  satisfied. Assume that $p_2=u_2$. 
  \bli
\item[{\upshape\bfseries (i)}] If $\ \frac{s_1- s_2}{d} >  \left(\frac{1}{p_1} -\frac{1 }{p_2}\right)_+$, then
\begin{equation*}%\label{morreyslarge}
    e_k \big(\id: \MAe(\Omega)\hookrightarrow \Az(\Omega)\big) \sim     k^{-\frac{s_1-s_2}{d}} .  
\end{equation*}
\item[{\upshape\bfseries (ii)}]
  If $\ \frac{p_1}{u_1} \Big(\frac{1}{p_1}- \frac{1}{p_2}\Big)< \frac{s_1- s_2}{d} \le  \frac{1}{p_1} -\frac{1 }{p_2}$, then there exists some $c>0$ and for any $\varepsilon>0$ some $c_\varepsilon>0$ such that for all $k\in\nat$, 
\begin{equation*}%\label{new_morrey_fu_a}
  c k^{- \frac{u_1}{u_1-p_1}(\frac{s_1-s_2}{d}-\frac{p_1}{u_1}(\frac{1}{p_1}- \frac{1}{p_2}) )} \le    e_k \big(\id\big) \le c_\varepsilon k^{- \frac{u_1}{u_1-p_1}(\frac{s_1-s_2}{d}-\frac{p_1}{u_1}(\frac{1}{p_1}- \frac{1}{p_2}) )+\varepsilon}.
  %k^{-\alpha (\frac{s_1-s_2}{d}-\frac{p_1}{u_1}(\frac{1}{p_1}- \frac{1}{p_2}) )+\varepsilon} ,  
\end{equation*}
%where $\alpha$ is given by \eqref{alpha}.
\eli
 %\begin{equation}\label{ek_intoA}
%    e_k \big(\id: \MAe(\Omega)\hookrightarrow \Az(\Omega)\big) \sim  k^{-\frac{s_1-s_2}{d}} .  
%\end{equation}
\end{corollary}

\begin{proof}
This follows from Theorems~\ref{ek_classical_fu} and \ref{reallymorrey_fu} with $p_2=u_2$.
\end{proof}

\begin{corollary}\label{C-ekintoLr}
 Let $\Omega$ be a bounded domain in $\R^d$, $1\leq r\leq\infty$, $0<p\leq u<\infty$, and $s>d \frac{p}{u}\left(\frac1p-\frac1r\right)_+$.
\bli
\item[{\upshape\bfseries (i)}] If $\ p\geq r$, then
\begin{equation}%\label{morreyslarge}
    e_k \big(\id: \MA(\Omega)\hookrightarrow L_r(\Omega)\big) \sim     k^{-\frac{s}{d}} .  
\end{equation}
\item[{\upshape\bfseries (ii)}]
  If $p<r$ and $\ d \frac{p}{u} \Big(\frac{1}{p}- \frac{1}{r}\Big)< s \le  d\Big(\frac{1}{p} -\frac{1 }{r}\Big)$, then there exists some $c>0$ and for any $\varepsilon>0$ some $c_\varepsilon>0$ such that for all $k\in\nat$, 
\begin{equation}\label{new_morrey_Lr}
  c k^{- \frac{u}{u-p}(\frac{s}{d}-\frac{p}{u}(\frac{1}{p}- \frac{1}{r}) )} \le    e_k \big(\id: \MA(\Omega)\hookrightarrow L_r(\Omega)\big) \le c_\varepsilon k^{- \frac{u}{u-p}(\frac{s}{d}-\frac{p}{u}(\frac{1}{p}- \frac{1}{r}) )+\varepsilon},
  %k^{-\alpha (\frac{s_1-s_2}{d}-\frac{p_1}{u_1}(\frac{1}{p_1}- \frac{1}{p_2}) )+\varepsilon} ,  
\end{equation}
\eli
where $L_\infty(\Omega)$ can be replaced in all cases by $C(\Omega)$.
\end{corollary}

%\begin{corollary}\label{ek-MorreyintoLinfty}
% Let $s\in\R$, $0<p < u <\infty$ and $q\in(0,\infty]$, with $s>\frac{d}{u}$. Then% embedding
%\begin{equation*}
%  e_k(\id: \MA(\Omega) \hookrightarrow L_{\infty}(\Omega)) \sim k^{-\frac{s}{d}}, \quad k\in\nat,
%  \end{equation*}
%\end{corollary}

\begin{proof}
The result follows from Corollary~\ref{ek_into_classical_fu} together with the well-known embeddings
%  \[
$  B^0_{r,1}(\Omega)\hookrightarrow L_r(\Omega)\hookrightarrow B^0_{r,\infty}(\Omega)$
%  \]
and the multiplicativity of entropy numbers \eqref{e-multi}.
\end{proof}

%\begin{ownremark}
%  Our result \cite[Cor.~3.3]{HaSk-krakow} already covered the case $r=\infty$ and $s>\frac{d}{p}$. If $\frac{d}{u}<s\leq 
%\frac{d}{p}$, then \eqref{new_morrey_Lr} reads as
%\[
%  c k^{- \frac{u}{u-p}(\frac{s}{d}-\frac{1}{u} )} \le    e_k \big(\id: \MA(\Omega)\hookrightarrow L_\infty(\Omega)\big) \le 
%c_\varepsilon k^{- \frac{u}{u-p}(\frac{s}{d}-\frac{1}{u} )+\varepsilon}.
  %k^{-\alpha (\frac{s_1-s_2}{d}-\frac{p_1}{u_1}(\frac{1}{p_1}- \frac{1}{p_2}) )+\varepsilon} ,  
%\]
%\end{ownremark}

\begin{ownremark}
We obtained in \cite{YHMSY,HaSk-krakow} some first results for approximation numbers of the embedding $\id_{\mathcal{A}}$.  In this context we refer to \cite[Sect.~6]{Bai-Si} where also the periodic case and more general Morrey type spaces are studied. 
\end{ownremark}

\begin{ownremark}
Thanks to one reviewer we were pointed to the recent preprint \cite{M-U} which is based on \cite{EN-2}. It might well be that using this approach one could circumvent our above interpolation argument and thus seal the $\varepsilon$-gap in the exponents in Theorems~\ref{reallymorreycase}(ii) and \ref{reallymorrey_fu}(ii), respectively. We leave it as an open question here.  
\end{ownremark}

\section*{Acknowledgment} 
We are  indebted to the referees of the first  version of that paper for their valuable remarks which helped to improve the results and the  presentation.   

%\subsection{Compact embeddings}

%%%%%%%%%%%%%%%%%%%%%%%%%%%%%%%%%%%%%%%%%%%%%%%%%%%%%%%%%%%%%%%%%%%%%%%%%%%%%%%%%%%%%%%%%%%%%%%%%%%%%%

\end{document}